\documentclass[preprint,12pt]{article}

\usepackage[margin=.875in]{geometry}

\usepackage{bm}

\usepackage{color}
\definecolor{red}{rgb}{.8,0,0}

\definecolor{blu}{rgb}{0,0,1}

\usepackage{mathtools}
\usepackage{amssymb}
\usepackage{amsmath,amsthm}
\usepackage{tikz}
\usepackage{tkz-graph}

\usepackage{graphicx}
\usetikzlibrary{arrows.meta,decorations.pathreplacing}

\usepackage[bb=boondox]{mathalfa} 
\usepackage[mathlines]{lineno}

\let\oldproofname=\proofname
\renewcommand{\proofname}{\rm\bf{\textit{\oldproofname}}}

\theoremstyle{definition}
\newtheorem{examplex}{Example}
\newenvironment{example}
  {\pushQED{\qed}\examplex}
  {\popQED\endexamplex}
  
\theoremstyle{remark}
\newtheorem{remarkx}{Remark}
\newenvironment{remark}
  {\pushQED{\qed}\remarkx}
  {\popQED\endremarkx}

\theoremstyle{plain}
\newtheorem{theorem}{Theorem}[section]
\newtheorem{corollary}[theorem]{Corollary}
\newtheorem{lemma}[theorem]{Lemma}
\newtheorem{proposition}[theorem]{Proposition}

\theoremstyle{definition}

\newtheorem*{note*}{Note}

\newcommand{\bzero}{\mathbb{0}} 

\newcommand{\bmu}{\bm{\mu}} 

\newcommand{\open}{``}

\newcommand{\diag}{\operatorname{diag}}
\newcommand{\diam}{\operatorname{diam}}

\newcommand{\dist}{\operatorname{dist}}

\newcommand{\conc}{\operatorname{conc}}

\newcommand{\bit}{\begin{itemize}}
\newcommand{\eit}{\end{itemize}}
\newcommand{\ben}{\begin{enumerate}}
\newcommand{\een}{\end{enumerate}}
\newcommand{\beq}{\begin{equation}}
\newcommand{\eeq}{\end{equation}}
\newcommand{\bea}{\begin{eqnarray*}}
\newcommand{\eea}{\end{eqnarray*}}
\newcommand{\bean}{\begin{eqnarray}}
\newcommand{\eean}{\end{eqnarray}}
\newcommand{\bpf}{\begin{proof}}
\newcommand{\epf}{\end{proof}}

\newcommand{\mc}{\mathcal}

\usepackage{amsfonts}
\DeclareSymbolFontAlphabet{\amsmathbb}{AMSb}

\allowdisplaybreaks

\begin{document}

\title{On  Kemeny's constant for trees with fixed order and diameter}

\author{
Lorenzo Ciardo\textsuperscript{a}\thanks{Corresponding author. E-mail: \tt lorenzci@math.uio.no}
,
Geir Dahl\textsuperscript{a}\thanks{E-mail: \tt geird@math.uio.no}
, and
Steve Kirkland\textsuperscript{b}\thanks{E-mail: \tt stephen.kirkland@umanitoba.ca}
}
\date{\textsuperscript{a} Department of Mathematics, University of Oslo, Oslo, Norway;\\
\textsuperscript{b} Department of Mathematics, University of Manitoba, Winnipeg, MB, Canada.}

\maketitle


\begin{abstract}
\noindent Kemeny's constant $\kappa(G)$ of a connected graph $G$ is a measure of the expected transit time for the random walk associated with $G$. In the current work, we consider the case when $G$ is a tree, and, in this setting, we provide lower and upper bounds for $\kappa(G)$ in terms of the order $n$ and diameter $\delta$ of $G$ by using two different techniques. The lower bound is given as Kemeny's constant of a particular caterpillar tree and, as a consequence, it is sharp. The upper bound is found via induction, by repeatedly removing pendent vertices from $G$. By considering a specific family of trees -- the broom-stars -- we show that the upper bound is asymptotically sharp.
\end{abstract}

\noindent{\bf Keywords.} 
Kemeny's constant; random walk on a graph; Markov chain; stochastic matrix; caterpillar

\medskip

\noindent{\bf AMS subject classifications.}
05C81, 
05C50, 
60J10, 
05C12, 
94C15  

\medskip\medskip


\section{Introduction}
In the information age, when people and ideas have virtually no physical barriers other than the ones of communicability, the role of information media is becoming increasingly  crucial. Having a direct control on the global flow of information is progressively difficult; an indirect control, however, is possible by acting on the architecture and geometry of networks transferring data (\cite{Meyn}). 

Markov chains are a widely used model for a physical entity moving in a network in discrete time steps. Consider the network as a simple nontrivial connected undirected graph $G=(V(G),E(G))$, where the vertices in $V(G)$ represent states and the edges in $E(G)$ represent connections between states. Given $i,j\in V(G)$, we denote the probability of going to $j$ in one step starting from $i$ by the real number $t_{ij}$, with $0\leq t_{ij}\leq 1$. The lack of a connection between $i$ and $j$ makes it impossible to go from $i$ to $j$ in one step, and this is reflected by choosing $t_{ij}=0$ if $\{i,j\}\not \in E(G)$. Moreover, we require that $\sum_{j\in V(G)}t_{ij}=1$ for $i\in V(G)$. Note that this description implies \textit{absence of memory}: the behavior of the system in the next time step does not depend on the complete history of the process, but only on the current state. The so-called \textit{transition matrix} $T=[t_{ij}]\in M_n$ (where $n=|V(G)|$) thus encodes the entire behavior of the Markov chain. As a consequence, it is possible to use linear algebraic techniques to predict the short and long term behavior of the system. Observe that $T$ is a real nonnegative row-stochastic matrix. If $T$ is \textit{irreducible} -- meaning that for any pair of indices $i,j$ there exists $k\in\amsmathbb{N}$ such that $(T^k)_{ij}>0$ -- by virtue of Perron-Frobenius theory there is a unique \textit{stationary distribution} $\textbf{w}=(w_i)\in\amsmathbb{R}^n$, which satisfies $\textbf{w}^TT=\textbf{w}^T$, $\textbf{w}>\bzero$ and $\textbf{w}^T\textbf{e}=1$ (where $\textbf{e}\in\amsmathbb{R}^n$ is the all ones vector). Notice, in particular, that $\textbf{w}$ is a probability distribution. If, in addition, $T$ is \textit{primitive} -- i.e., $\exists k \in \amsmathbb{N}$ such that $T^k>0$ -- then the Markov chain will converge to $\textbf{w}$ regardless of the initial probability distribution. As a consequence, we say that $\textbf{w}$ encodes the long-term behavior of the system. A way to look at the short-term behavior is to consider the so--called \textit{mean first passage matrix} $M=[m_{ij}]\in M_n$, where $m_{ij}$ is the expected number of steps needed to reach state $j$ for the first time starting from state $i$. For background on the theory of Markov chains we refer the reader to \cite{KemenySnell} and \cite{Seneta}.

 A meaningful indicator of the communicability in a network arises by combining long-term and short-term behavior of the associated Markov chain. It was shown in \cite{KemenySnell} that the quantity 
\begin{equation*}
\kappa_i(T)\coloneqq (M\textbf{w})_i-1=\sum_{j=1}^nm_{ij}w_j-1
\end{equation*}
is not dependent on the choice of $i$, and it is thus a constant for the Markov chain. This quantity is called \textit{Kemeny's constant} and it is denoted by $\kappa(T)$. As, clearly,
\begin{equation*}
\kappa(T)=\textbf{w}^TM\textbf{w}-1=\sum_{i,j=1}^nm_{ij}w_iw_j-1,
\end{equation*}
we see that $\kappa(T)$ measures the expected travel time between two randomly chosen states, sampled accordingly to the stationary distribution $\textbf{w}$\footnote{The term \open $-1$" in the definitions of $\kappa_i(T)$ and $\kappa(T)$ is convenient in order to yield the following expression for  Kemeny's constant in terms of the spectrum $\sigma(T)=\{1,\lambda_2,\lambda_3,\dots,\lambda_n\}$ of the transition matrix $T$: $\kappa(T)=\sum_{j=2}^n\frac{1}{1-\lambda_j}$ (see \cite{KemenySnell}).}. We can then see  Kemeny's constant as a  network metric that measures the long--run ability to transmit information: the smaller $\kappa(T)$ is, the faster information can spread in the network (\cite{Breen,Catral,Hunter}). As a consequence, one can control the long--run diffusion rate of the information flow by performing modifications on the network which lead to the desired change in the value of  Kemeny's constant. Particularly interesting in this regard, is the phenomenon known as the \textit{Braess' paradox for graphs}, which occurs when adding a new connection in the network has the counter--intuitive effect of \textit{increasing} the value of  Kemeny's constant instead of decreasing it (\cite{Ciardo_braess,HuKirkland,KirklandZeng}). 

In the present work, we focus on a particular kind of Markov chain on $G$, where, in each step, a random walker moves from a vertex $i$ (the current position) to one of the neighbors of $i$ with all the neighbors being equally likely (see \cite{Lovasz} for a survey on this type of random walk). Let $A$ be the adjacency matrix of $G$ and $D=\diag(A\textbf{e})$ be its diagonal degree matrix. We observe that, in this case, the transition matrix is given by $T=D^{-1}A$ (which is irreducible since $G$ is connected). For such a Markov chain the stationary distribution, the mean first passage matrix and  Kemeny's constant are each determined once the graph $G$ is fixed. Hence, we can write $\kappa(G)$ instead of $\kappa(T)$. The problem of understanding how the structure of the network influences the value of  Kemeny's constant then  becomes  entirely graph--theoretical.

The aim of the current work is to exhibit a lower and an upper bound for  Kemeny's constant for a tree -- i.e., an acyclic connected graph -- in terms of its order $n$ and its diameter $\delta$. Following the interpretation given above, this provides knowledge about the long--term spread of information  in a network once two basic quantities of the network -- the number of nodes and the maximum distance -- are known. In Section \ref{sec_recursive_formula_kemeny_constant_tree} we give a recursive formula for  Kemeny's constant of a tree in terms of  Kemeny's constant of certain subtrees. Both lower and upper bounds make use of this formula, but they are built following two different approaches. For the former, we explicitly exhibit the minimizers of  Kemeny's constant among the trees having fixed order and diameter, and we show that they belong to the class of so-called \textit{caterpillars}. The recursive formula mentioned above is particularly simple in the case of caterpillars, and this allows us to obtain a (sharp) lower bound as the explicit expression for  Kemeny's constant of the minimizers. This is done in Sections \ref{sec_extremal_caterpillars} and \ref{sec_lower_bound_kememy}. In contrast, the upper bound is obtained by using the recursive formula in an inductive argument (Section \ref{sec_upper_bound}). Finally, in Section \ref{sec_asymptotic_analysis_upper_bound} we consider a family of trees -- the \textit{broom-stars} -- that attain a particularly large value for Kemeny's constant. We use this family to show that the upper bound is asymptotically sharp.

\medskip	
\noindent\underline{Notation}:
We  let $\amsmathbb{R}^n$ denote the space of $n$-dimensional real column vectors, and we identify such vectors with the corresponding $n$-tuples.
The $i$--th unit vector in $\amsmathbb{R}^n$ is denoted by $\textbf{e}_i$, and the all  ones vector in $\amsmathbb{R}^n$ is denoted by $\textbf{e}$. $M_n$ denotes the space of real square matrices of order $n$, and $M_{n_1,n_2}$ denotes the space of real $n_1\times n_2$ matrices. The vertex set and the edge set of a graph $G$ are denoted by $V(G)$ and $E(G)$, respectively. The \textit{order} and the \textit{size} of $G$ are $|V(G)|$ and $|E(G)|$, respectively. The \textit{trivial graph} is the graph having one vertex. If $v\in V(G)$ is a \textit{pendent vertex} (i.e., a vertex of degree $1$), $G-v$ denotes the graph obtained from $G$ by removing $v$ and the unique edge incident to $v$. We identify two graphs when they are isomorphic. We denote the \textit{star tree} having $n$ vertices by $S(n)$. A \textit{rooted tree} $(T,r)$ is a tree with a distinguished vertex (\textit{root}) $r\in V(T)$. When the root $r$ is clear or not relevant in the context, we write $T$ instead of $(T,r)$. 

\section{Concatenation of trees}
\label{sec_recursive_formula_kemeny_constant_tree}
Given a tree $T$ having $n$ vertices, let $\textbf{d}\in\amsmathbb{R}^n$ be its \textit{degree vector}, whose $i$--th entry $d_i$ is the degree of vertex $i$. Also, let $\Delta\in M_n$ be its \textit{distance matrix}, whose $ij$--th entry $\Delta_{ij}$ is the number of edges in the path connecting vertex $i$ and vertex $j$. In \cite[Theorem 3.1.]{KirklandZeng} we find the following combinatorial formula for  Kemeny's constant $\kappa(T)$ of a nontrivial tree $T$:
\begin{equation}
\label{eq_kemeny_constant_tree_1723_05_nov}
\kappa(T)=\frac{\textbf{d}^T\Delta\textbf{d}}{4(n-1)}
\end{equation}
(by convention, we set  Kemeny's constant of the trivial tree to be equal to $0$).

In this section we use \eqref{eq_kemeny_constant_tree_1723_05_nov} to provide an expression for  Kemeny's constant of a tree obtained by concatenating some given rooted trees. In the rest of the paper this will be used to find  lower and upper bounds  on   Kemeny's constant for trees with fixed order and diameter. 

Let $(R,r)$ be a rooted tree. We define its \textit{moment} $\mu(R,r)$ as follows:
\begin{equation}
\label{def_moment_of_inertia_1501_01_nov}
\mu(R,r)=\sum_{v\in V(R)}\dist(v,r)\deg(v).
\end{equation}
The motivation for this name comes from the homonymous notion in mechanics. Given a force $\textbf{F}$ applied to a point particle having position $\textbf{P}$ with respect to a fixed point (\open fulcrum"), the moment or \textit{torque} of $\textbf{F}$ is the vector $\textbf{P}\times\textbf{F}$. The moment of a set of forces $\textbf{F}_1,\textbf{F}_2,\dots,\textbf{F}_n$, each one applied to a point particle having position $\textbf{P}_1,\textbf{P}_2,\dots,\textbf{P}_n$, is simply $\sum_{i=1}^n\textbf{P}_i\times\textbf{F}_i$. Let us now draw the tree $R$ by arranging its vertices horizontally, starting from the root $r$ (which we consider as the fixed fulcrum). Suppose, also, that the degree of a vertex represents its weight. Then, formula \eqref{def_moment_of_inertia_1501_01_nov} gives the (magnitude of) the moment of the gravity force acting on the vertices of $R$ (Figure \ref{fig_tree_with_moment}).
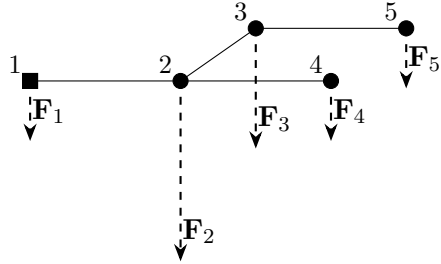
\begin{figure}[h]
\centering
\begin{tikzpicture}[scale=1]
%
\draw [fill,draw=black] (-.1,-.1) rectangle ++(.2,.2);
\draw[fill]  (2,0) circle (0.1cm);
\draw[fill]  (4,0) circle (0.1cm);
\draw[fill]  (3,.7) circle (0.1cm);
\draw[fill]  (5,.7) circle (0.1cm);
\draw  (-0.2,0.225) node{\footnotesize{$1$}};
\draw  (2-0.2,0.225) node{\footnotesize{$2$}};
\draw  (4-0.2,0.225) node{\footnotesize{$4$}};
\draw  (3-0.2,.7+0.225) node{\footnotesize{$3$}};
\draw  (5-0.2,.7+0.225) node{\footnotesize{$5$}};
%
%
\draw  (.25,-.4) node{\small{$\textbf{F}_1$}};
\draw  (2.25,-2) node{\small{$\textbf{F}_2$}};
\draw  (4.25,-.4) node{\small{$\textbf{F}_4$}};
\draw  (3.25,.7-1.2) node{\small{$\textbf{F}_3$}};
\draw  (5.25,.7-.4) node{\small{$\textbf{F}_5$}};
%
\draw[black] (0,0) -- (2,0);
\draw[black] (2,0) -- (4,0);
\draw[black] (2,0) -- (3,.7);
\draw[black] (3,.7) -- (5,.7);
%
%
\draw[-{Stealth[scale=1]},thick,shorten >=0pt,shorten <=0pt,dashed] (0,0) -- (0,-.8*1);
\draw[-{Stealth[scale=1]},thick,shorten >=0pt,shorten <=0pt,dashed] (2,0) -- (2,-.8*3);
\draw[-{Stealth[scale=1]},thick,shorten >=0pt,shorten <=0pt,dashed] (4,0) -- (4,-.8*1);
\draw[-{Stealth[scale=1]},thick,shorten >=0pt,shorten <=0pt,dashed] (3,.7) -- (3,.7-.8*2);
\draw[-{Stealth[scale=1]},thick,shorten >=0pt,shorten <=0pt,dashed] (5,.7) -- (5,.7-.8*1);
\end{tikzpicture}
\caption{\footnotesize{A rooted tree having vertex $1$ as root. The weight of each vertex is its degree, so that the moment is $\dist(1,1)|\textbf{F}_1|+\dist(2,1)|\textbf{F}_2|+\dist(3,1)|\textbf{F}_3|+\dist(4,1)|\textbf{F}_4|+\dist(5,1)|\textbf{F}_5|=0\cdot 1+1\cdot 3+2\cdot 2+2\cdot 1+3\cdot 1=12$.}}
\label{fig_tree_with_moment}
\end{figure}
\begin{lemma}
\label{lemma_moment_star}
Let $(R,r)$ be a rooted tree of order $n$. Then
\begin{equation*}
\mu(R,r)\geq n-1,
\end{equation*}
with equality if and only if $R$ is the star $S(n)$ and $r$ is the central vertex of $S(n)$.
\end{lemma}
\begin{proof}
From the definition \eqref{def_moment_of_inertia_1501_01_nov} of moment we have
\begin{align*}
\mu(R,r)&=\sum_{v\in V(R)}\dist(v,r)\deg(v)
=\sum_{\substack{v\in V(R)\\v\neq r}}\dist(v,r)\deg(v)
\geq \sum_{\substack{v\in V(R)\\v\neq r}}1\cdot 1
=n-1.
\end{align*}
Equality holds precisely when all the vertices except the root $r$ have degree $1$ and distance $1$ from $r$, i.e., when $R$ is the star $S(n)$ and $r$ is the central vertex of $S(n)$.
\end{proof}
Given an integer $k\geq 1$, consider $k$ rooted trees $(T_1,r_1),(T_2,r_2),\dots,(T_k,r_k)$. For each ${\ell}=1,2,\dots,k$ let $n_{\ell}=|V(T_{\ell})|$ be the order of $T_{\ell}$, $m_{\ell}=|E(T_{\ell})|=n_{\ell}-1$ be the size of $T_{\ell}$, $\textbf{d}^{({\ell})}=(d^{({\ell})}_i)\in\amsmathbb{R}^{n_{\ell}}$ be the degree vector of $T_{\ell}$, $\Delta^{({\ell})}=[\Delta^{({\ell})}_{ij}]\in M_{n_{\ell}}$ be the distance matrix of $T_{\ell}$ and $\mu_{\ell}=\mu(T_{\ell},r_{\ell})=\textbf{e}_1^T\Delta^{({\ell})}\textbf{d}^{({\ell})}$ be the moment of $(T_{\ell},r_{\ell})$ (where the vertices of $T_{\ell}$ are labeled in such a way that vertex $1$ corresponds to the root $r_{\ell}$). 
We also define the \textit{size vector} $\textbf{m}=(m_1,m_2,\dots,m_k)\in\amsmathbb{R}^k$ and the \textit{moment vector} $\bmu=(\mu_1,\mu_2,\dots,\mu_k)\in\amsmathbb{R}^k$. Consider the tree $T=\conc((T_1,r_1),(T_2,r_2),\dots,(T_k,r_k))$ obtained by taking the disjoint union of $T_1,T_2,\dots,T_k$ and adding the edges $r_1r_2,r_2r_3,\dots,r_{k-1}r_k$ (Figure \ref{fig_concatenation_trees}), and let its order be $n=\sum_{{\ell}=1}^kn_\ell$. We shall refer to $T$ as to the \textit{concatenation} of $(T_1,r_1),(T_2,r_2),\dots,(T_k,r_k)$. Finally, we introduce the symmetric Toeplitz matrix $X=[x_{ij}]\in M_k$ defined by $x_{ij}=|i-j|$ $(i,j \le k)$.
\begin{figure}[h]
\centering
\begin{tikzpicture}[scale=.8]
%
\draw  (0,.8) node{$T_1$};
\draw(.35,-.25) node{\footnotesize{$r_1$}};
\draw [fill,draw=black] (-.1,-.1) rectangle ++(.2,.2);
\draw[fill]  (0,-2) circle (0.1cm);
\draw[fill] (-1,-2-1.7321) circle (0.1cm);
\draw[fill] (1,-2-1.7321) circle (0.1cm);
\draw[black] (0,0) -- (0,-2);
\draw[black] (0,-2) -- (-1,-2-1.7321);
\draw[black] (0,-2) -- (1,-2-1.7321);
%
%
\begin{scope}[shift={(2.5,0)}]
\draw  (0,0.8) node{$T_2$};
\draw(.35,-.25) node{\footnotesize{$r_2$}};
\draw [fill,draw=black] (0-.1,-.1) rectangle ++(.2,.2);
\end{scope}
%
%
\begin{scope}[shift={(5,0)}]
\draw  (0,0.8) node{$T_3$};
\draw(.35,-.25) node{\footnotesize{$r_3$}};
\draw [fill,draw=black] (0-.1,-.1) rectangle ++(.2,.2);
\draw[fill] (-1,-1.7321) circle (0.1cm);
\draw[fill] (1,-1.7321) circle (0.1cm);
\draw[fill]  (0,-2) circle (0.1cm);
\draw[black] (0,0) -- (-1,-1.7321);
\draw[black] (0,0) -- (1,-1.7321);
\draw[black] (0,0) -- (0,-2);
\end{scope}
%
%
\begin{scope}[shift={(7.5,0)}]
\draw  (0,0.8) node{$T_4$};
\draw(.35,-.25) node{\footnotesize{$r_4$}};
\draw [fill,draw=black] (0-.1,-.1) rectangle ++(.2,.2);
\draw[fill]  (0,-2) circle (0.1cm);
\draw[fill]  (0,-4) circle (0.1cm);
\draw[black] (0,0) -- (0,-2);
\draw[black] (0,-2) -- (0,-4);
\end{scope}
%
%
%
\begin{scope}[shift={(12,0)},scale=.7]
\draw  (3,0.8+.5) node{$T$};
\draw [fill] (0,0) circle (0.1cm);
\draw[fill]  (0,-2) circle (0.1cm);
\draw[fill] (-1,-2-1.7321) circle (0.1cm);
\draw[fill] (1,-2-1.7321) circle (0.1cm);
\draw [fill] (2,0) circle (0.1cm);
\draw [fill] (4,0) circle (0.1cm);
\draw[fill] (4-1,-1.7321) circle (0.1cm);
\draw[fill] (4+1,-1.7321) circle (0.1cm);
\draw[fill]  (4,-2) circle (0.1cm);
\draw [fill] (6,0) circle (0.1cm);
\draw [fill] (6,-2) circle (0.1cm);
\draw [fill] (6,-4) circle (0.1cm);
\draw[black] (0,0) -- (0,-2);
\draw[black] (0,-2) -- (-1,-2-1.7321);
\draw[black] (0,-2) -- (1,-2-1.7321);
\draw[black] (4,0) -- (4-1,-1.7321);
\draw[black] (4,0) -- (4+1,-1.7321);
\draw[black] (4,0) -- (4,-2);
\draw[black] (6,0) -- (6,-2);
\draw[black] (6,-2) -- (6,-4);
\draw[black] (0,0) -- (2,0);
\draw[black] (2,0) -- (4,0);
\draw[black] (4,0) -- (6,0);
\end{scope}
\end{tikzpicture}
\caption{\footnotesize{Four rooted trees and their concatenation $T=\conc((T_1,r_1),(T_2,r_2),(T_3,r_3),(T_4,r_4))$.}}
\label{fig_concatenation_trees}
\end{figure}
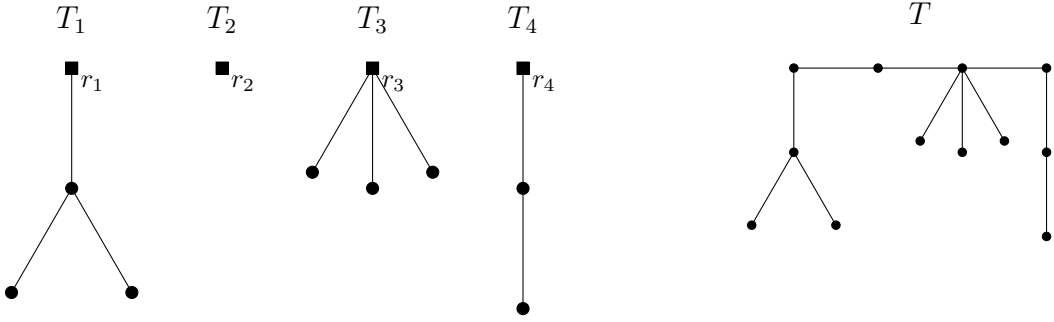
\begin{proposition}
\label{prop_kemeny_recursive_22_oct_2019}
Let $T=\conc((T_1,r_1),(T_2,r_2),\dots,(T_k,r_k))$. If $T$ is nontrivial, its Kemeny's constant may be expressed as
\begin{equation*}
\resizebox{1\hsize}{!}{$\displaystyle
\kappa(T)=\frac{1}{n-1}\left(\sum_{i=1}^km_i\kappa(T_i)
+\textbf{m}^TX(\textbf{m}+2\textbf{e})+(n-1)\bmu^T\textbf{e}
-{\bmu}^T\textbf{m}
+\frac{k^3}{3}-kn+n+\frac{k}{6}-\frac{1}{2}\right).
$}
\end{equation*}
\end{proposition}
\begin{proof}
If $k=1$, a straightforward computation shows that the result holds, so assume that $k\geq 2$. Denote the degree vector of $T$ by $\textbf{d}\in\amsmathbb{R}^n$. Using a suitable ordering of vertices in $T$ and defining
\begin{align*}
\tilde{\textbf{d}}=\begin{bmatrix}
\textbf{d}^{(1)} \\ 
\textbf{d}^{(2)} \\ 
\vdots\\
\textbf{d}^{(k-1)}\\
\textbf{d}^{(k)}
\end{bmatrix},
\hspace{.5cm} 
\textbf{g}=
\begin{bmatrix}
\textbf{e}_1\\
\textbf{e}_1\\
\vdots\\
\textbf{e}_1\\
\textbf{e}_1
\end{bmatrix},
\hspace{.5cm}
\textbf{h}=
\begin{bmatrix}
\textbf{e}_1\\
\bzero\\
\vdots\\
\bzero\\
\textbf{e}_1
\end{bmatrix},
\end{align*}
we see that 
\begin{align*}
\textbf{d}=\tilde{\textbf{d}} + 2\textbf{g} -\textbf{h}.
\end{align*}
We partition the distance matrix $\Delta=[\Delta_{ij}]\in M_n$ of $T$ as follows:
\begin{align*}
\Delta=\begin{bmatrix}
\Delta^{(11)} & \Delta^{(12)} & \cdots & \Delta^{(1k)}\\
\Delta^{(21)} & \Delta^{(22)} & \cdots & \Delta^{(2k)}\\
\vdots & \vdots & \ddots & \vdots\\
\Delta^{(k1)} & \Delta^{(k2)} & \cdots & \Delta^{(kk)}\\
\end{bmatrix}
\end{align*}
where $\Delta^{(ij)}\in M_{n_i,n_j}$. Given two vertices $u,v\in V(T)$ and supposing that $u$ and $v$ both belong to $V(T_i)$ for some $i\in\{1,2,\dots,k\}$, we see that $\Delta_{uv}=\Delta^{(i)}_{uv}$. This shows that $\Delta^{(ii)}=\Delta^{(i)}$ for $i\in\{1,2,\dots,k\}$. On the other hand, if $u\in V(T_{i})$ and $v\in V(T_{j})$ with $i\neq j$, then
\begin{align*}
\Delta_{uv}=\Delta^{(i)}_{u1}+\Delta^{(j)}_{1v}+|i-j|=\textbf{e}_u^T\Delta^{(i)}\textbf{e}_1+\textbf{e}_1^T\Delta^{(j)}\textbf{e}_v+|i-j|
\end{align*}
and hence
\begin{align*} 
\Delta^{(ij)}=\Delta^{(i)}\textbf{e}_1\textbf{e}^T+\textbf{e}\textbf{e}_1^T\Delta^{(j)}+|i-j|\textbf{e}\textbf{e}^T.
\end{align*}
We obtain
\begin{align*}
\textbf{d}^T\Delta\textbf{d}&=(\tilde{\textbf{d}}^T+2\textbf{g}^T-\textbf{h}^T)\Delta
(\tilde{\textbf{d}}+2\textbf{g}-\textbf{h})\\
&=\tilde{\textbf{d}}^T\Delta\tilde{\textbf{d}}+4\textbf{g}^T\Delta\textbf{g}+\textbf{h}^T\Delta
\textbf{h}+4\tilde{\textbf{d}}^T\Delta\textbf{g}-2\tilde{\textbf{d}}^T\Delta\textbf{h}
-4\textbf{g}^T\Delta\textbf{h}.
\end{align*}
Observe that
\begin{align*}
&\bullet&
\tilde{\textbf{d}}^T\Delta\tilde{\textbf{d}}&=
\sum_{i,j=1}^k{\textbf{d}^{(i)}}^T\Delta^{(ij)}\textbf{d}^{(j)}\\
&&&=\sum_{i=1}^k{\textbf{d}^{(i)}}^T\Delta^{(i)}\textbf{d}^{(i)}+\sum_{i\neq j}{\textbf{d}^{(i)}}^T(\Delta^{(i)}\textbf{e}_1\textbf{e}^T+\textbf{e}\textbf{e}_1^T\Delta^{(j)}+|i-j|\textbf{e}\textbf{e}^T)\textbf{d}^{(j)}\\
&&&=\sum_{i=1}^k 4m_i\kappa(T_i)
+
\sum_{i\neq j}(2\mu_im_j+2\mu_jm_i+4m_im_j|i-j|)\\
&&&=\sum_{i=1}^k4m_i\kappa(T_i)
+
4\sum_{i\neq j}\mu_im_j+4\sum_{i,j=1}^km_im_j|i-j|\\
&&&=\sum_{i=1}^k4m_i\kappa(T_i)
+
4(\textbf{e}^T{\bmu}\textbf{m}^T\textbf{e}-
{\bmu}^T\textbf{m})+4\textbf{m}^TX\textbf{m}\\
&&&=\sum_{i=1}^k4m_i\kappa(T_i)
+
4(n-k){\bmu}^T\textbf{e}-4{\bmu}^T\textbf{m}+4\textbf{m}^TX\textbf{m}.
\\
&\bullet& \textbf{g}^T\Delta\textbf{g}&=\sum_{i,j=1}^k\textbf{e}_1^T\Delta^{(ij)}\textbf{e}_1=
\sum_{i\neq j}\textbf{e}_1^T(\Delta^{(i)}\textbf{e}_1\textbf{e}^T+\textbf{e}\textbf{e}_1^T\Delta^{(j)}+|i-j|\textbf{e}\textbf{e}^T)\textbf{e}_1\\&&&=\sum_{i\neq j}|i-j|=\frac{k^3-k}{3}.\\
&\bullet&
\textbf{h}^T\Delta\textbf{h}&=\textbf{e}_1^T\Delta^{(11)}\textbf{e}_1
+\textbf{e}_1^T\Delta^{(kk)}\textbf{e}_1+2\textbf{e}_1^T\Delta^{(1k)}\textbf{e}_1\\
&&&=
2\textbf{e}_1^T(\Delta^{(1)}\textbf{e}_1\textbf{e}^T+\textbf{e}\textbf{e}_1^T\Delta^{(k)}+|1-k|\textbf{e}\textbf{e}^T)\textbf{e}_1=2|1-k|=2(k-1).\\
&\bullet&
\tilde{\textbf{d}}^T\Delta\textbf{g}&=\sum_{i,j=1}^k {\textbf{d}^{(i)}}^T\Delta^{(ij)}\textbf{e}_1\\
&&&=\sum_{i=1}^k{\textbf{d}^{(i)}}^T\Delta^{(i)}\textbf{e}_1
+\sum_{i\neq j}{\textbf{d}^{(i)}}^T\Delta^{(ij)}\textbf{e}_1\\
&&&=
\sum_{i=1}^k\mu_i
+\sum_{i\neq j}{\textbf{d}^{(i)}}^T(\Delta^{(i)}\textbf{e}_1\textbf{e}^T+\textbf{e}\textbf{e}_1^T\Delta^{(j)}+|i-j|\textbf{e}\textbf{e}^T)\textbf{e}_1\\
&&&={\bmu}^T\textbf{e}+\sum_{i\neq j}(\mu_i+2m_i|i-j|)\\
&&&={\bmu}^T\textbf{e}+(k-1){\bmu}^T\textbf{e}+2\sum_{i,j=1}^km_i|i-j|\\
&&&=k{\bmu}^T\textbf{e}+2\textbf{m}^TX\textbf{e}.\\
&\bullet&
\tilde{\textbf{d}}^T\Delta\textbf{h}&=
\sum_{i=1}^k\left({\textbf{d}^{(i)}}^T\Delta^{(i1)}\textbf{e}_1
+{\textbf{d}^{(i)}}^T\Delta^{(ik)}\textbf{e}_1\right)\\
&&&={\textbf{d}^{(1)}}^T\Delta^{(1)}\textbf{e}_1+\sum_{i=2}^k{\textbf{d}^{(i)}}^T\Delta^{(i1)}\textbf{e}_1
+{\textbf{d}^{(k)}}^T\Delta^{(k)}\textbf{e}_1+
\sum_{i=1}^{k-1}{\textbf{d}^{(i)}}^T\Delta^{(ik)}\textbf{e}_1\\
&&&=\mu_1+\mu_k+\sum_{i=2}^k{\textbf{d}^{(i)}}^T(\Delta^{(i)}\textbf{e}_1\textbf{e}^T+
\textbf{e}\textbf{e}_1^T\Delta^{(1)}+(i-1)\textbf{e}\textbf{e}^T)\textbf{e}_1\\
&&&+
\sum_{i=1}^{k-1}{\textbf{d}^{(i)}}^T(\Delta^{(i)}\textbf{e}_1\textbf{e}^T+\textbf{e}\textbf{e}_1^T\Delta^{(k)}
+(k-i)\textbf{e}\textbf{e}^T)\textbf{e}_1\\
&&&=\mu_1+\mu_k+\sum_{i=2}^{k}(\mu_i+2m_i(i-1))+\sum_{i=1}^{k-1}(\mu_i+2m_i(k-i))\\
&&&=2{\bmu}^T\textbf{e}+\sum_{i=2}^{k-1}2m_i(k-1)+2m_k(k-1)+2m_1(k-1)\\
&&&=2{\bmu}^T\textbf{e}+2(k-1)\sum_{i=1}^km_i\\
&&&=2{\bmu}^T\textbf{e}+2(k-1)(n-k).\\
&\bullet&
\textbf{g}^T\Delta\textbf{h}&=\sum_{i=1}^k\left(\textbf{e}_1^T\Delta^{(i1)}\textbf{e}_1
+\textbf{e}_1^T\Delta^{(ik)}\textbf{e}_1\right)\\
&&&=\textbf{e}_1^T\Delta^{(1)}\textbf{e}_1+\sum_{i=2}^k\textbf{e}_1^T\Delta^{(i1)}\textbf{e}_1
+\textbf{e}_1^T\Delta^{(k)}\textbf{e}_1+
\sum_{i=1}^{k-1}\textbf{e}_1^T\Delta^{(ik)}\textbf{e}_1\\
&&&=
\sum_{i=2}^k\textbf{e}_1^T(\Delta^{(i)}\textbf{e}_1\textbf{e}^T+\textbf{e}\textbf{e}_1^T\Delta^{(1)}
+(i-1)\textbf{e}\textbf{e}^T)\textbf{e}_1\\
&&&+\sum_{i=1}^{k-1}\textbf{e}_1^T(\Delta^{(i)}\textbf{e}_1\textbf{e}^T+\textbf{e}\textbf{e}_1^T\Delta^{(k)}+(k-i)\textbf{e}\textbf{e}^T)\textbf{e}_1\\
&&&=\sum_{i=2}^k(i-1)+\sum_{i=1}^{k-1}(k-i)\\
&&&=\sum_{i=1}^{k}(i-1)+\sum_{i=1}^k(k-i)\\
&&&=\sum_{i=1}^k(k-1)\\
&&&=k^2-k.
\end{align*}
Putting these observations all together, we obtain
\begin{align*}
&\textbf{d}^T\Delta\textbf{d}=\sum_{i=1}^k4m_i\kappa(T_i)
+
4(n-k){\bmu}^T\textbf{e}-4{\bmu}^T\textbf{m}+4\textbf{m}^TX\textbf{m}+\frac{4}{3}(k^3-k)\\
&\hspace{1.078cm}+2(k-1)+4k{\bmu}^T\textbf{e}+8\textbf{m}^TX\textbf{e}-4{\bmu}^T\textbf{e}-4(k-1)(n-k)-4k^2+4k\\
&=\sum_{i=1}^k4m_i\kappa(T_i)
+
4\textbf{m}^TX(\textbf{m}+2\textbf{e})+4(n-1){\bmu}^T\textbf{e}-4{\bmu}^T\textbf{m}
+\frac{4}{3}k^3-4kn+4n+\frac{2}{3}k-2.
\end{align*}
Applying \eqref{eq_kemeny_constant_tree_1723_05_nov} yields the desired expression for $\kappa(T)$.
\end{proof}
A first, straightforward application of Proposition \ref{prop_kemeny_recursive_22_oct_2019} allows us to express  Kemeny's constant of a tree $T$ in terms of that of $T-v$ where $v$ is a pendent vertex. 
\begin{corollary}
\label{cor_kemeny_pendent_vertex}
Let $T$ be a nontrivial tree of order $n$, let $v$ be a pendent vertex and let $w$ be the neighbor of $v$. Then
\begin{equation*}
\kappa(T)=\frac{1}{n-1}\left((n-2)\kappa(T-v)+\mu(T-v,w)+n-\frac{3}{2}\right).
\end{equation*}
\end{corollary}
\begin{proof}
Consider the trees $T_1$ and $T_2$, where $T_1$ is the trivial tree on vertex set $\{v\}$ and $T_2=T-v$, and notice that $T$ is the concatenation of the rooted trees $(T_1,v)$ and $(T_2,w)$. Denote the size vector and the moment vector of the concatenation by $\textbf{m}=(m_1,m_2)$ and ${\bmu}=(\mu_1,\mu_2)$, respectively, and notice that $m_1=0$, $m_2=n-2$, $\mu_1=0$ and $\mu_2=\mu(T-v,w)$. Using Proposition \ref{prop_kemeny_recursive_22_oct_2019}, we obtain
\begin{small}
\begin{align*}
\kappa(T)&=\frac{1}{n-1}\left({\displaystyle\sum_{i=1}^2m_i}\kappa(T_i)
+\textbf{m}^TX(\textbf{m}+2\textbf{e})+(n-1){\bmu}^T\textbf{e}
-{\bmu}^T\textbf{m}
+\frac{8}{3}-2n+n+\frac{1}{3}-\frac{1}{2}\right)\notag\\
&=\frac{1}{n-1}\left((n-2)\kappa(T_2)
+\textbf{m}^TX(\textbf{m}+2\textbf{e})+(n-1){\bmu}^T\textbf{e}
-{\bmu}^T\textbf{m}
-n+\frac{5}{2}\right).
\end{align*}
\end{small}
The result follows by observing that
\begin{equation*}
\textbf{m}^TX(\textbf{m}+2\textbf{e})=
\begin{bmatrix}
0 &
n-2
\end{bmatrix}
\begin{bmatrix}
0 & 1 \\
1 & 0
\end{bmatrix}
\begin{bmatrix}
2 \\
n
\end{bmatrix}
=2n-4
\end{equation*}
and
\begin{equation*}
(n-1){\bmu}^T\textbf{e}-{\bmu}^T\textbf{m}
={\bmu}^T((n-1)\textbf{e}-\textbf{m})
=
\begin{bmatrix}
0 & \mu(T-v,w)
\end{bmatrix}
\begin{bmatrix}
n-1 \\
1
\end{bmatrix}
=
\mu(T-v,w).\qedhere
\end{equation*}
\end{proof}
\section{Kemeny's constant for caterpillars}
\label{sec_extremal_caterpillars}
Given a positive integer $k$ and a nonnegative integer vector $\textbf{p}=(p_1, p_2, \ldots,p_k)$, we define the \textit{caterpillar} $C_k(\textbf{p})$ to be the tree consisting of a path $c_1 c_2 \dots c_k$ of $k$ vertices, called the \textit{central vertices}, and $p_i$ pendent vertices attached to $c_i$ for each $i\leq k$ (Figure \ref{fig_caterpillar}). Observe that the order of $C_k(\textbf{p})$ is $n=k+\sum_{i=1}^kp_i$. 
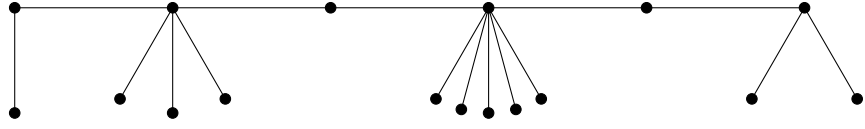
\begin{figure}[h]
\centering
\begin{tikzpicture}[scale=.7]
%
\draw [fill] (0,0) circle (0.1cm);
	\draw [fill] (0,-2) circle (0.1cm);
\draw [fill] (3,0) circle (0.1cm);
	\draw[fill]  (3,-2) circle (0.1cm);
	\draw[fill] (3-1,-1.7321) circle (0.1cm);
	\draw[fill] (3+1,-1.7321) circle (0.1cm);
\draw [fill] (6,0) circle (0.1cm);
\draw [fill] (9,0) circle (0.1cm);
	\draw[fill]  (9,-2) circle (0.1cm);
	\draw[fill] (9-1,-1.7321) circle (0.1cm);
	\draw[fill] (9+1,-1.7321) circle (0.1cm);
	\draw[fill] (9+0.5176,-1.9319) circle (0.1cm);
	\draw[fill] (9-0.5176,-1.9319) circle (0.1cm);
\draw [fill] (12,0) circle (0.1cm);
\draw [fill] (15,0) circle (0.1cm);
	\draw[fill] (15-1,-1.7321) circle (0.1cm);
	\draw[fill] (15+1,-1.7321) circle (0.1cm);
\draw[black] (0,0) -- (3,0) -- (6,0) -- (9,0) -- (12,0) -- (15,0);
\draw[black] (0,0) -- (0,-2);
\draw[black] (3,0) -- (3,-2);
\draw[black] (3,0) -- (3-1,-1.7321);
\draw[black] (3,0) -- (3+1,-1.7321);
\draw[black] (9,0) -- (9,-2);
\draw[black] (9,0) -- (9-1,-1.7321);
\draw[black] (9,0) -- (9+1,-1.7321);
\draw[black] (9,0) -- (9+0.5176,-1.9319);
\draw[black] (9,0) -- (9-0.5176,-1.9319);
\draw[black] (15,0) -- (15-1,-1.7321);
\draw[black] (15,0) -- (15+1,-1.7321);
\end{tikzpicture}
\caption{\footnotesize{The caterpillar $C_6(\textbf{p})$ with $\textbf{p}=(1,3,0,5,0,2)$.}}
\label{fig_caterpillar}
\end{figure}

From Proposition \ref{prop_kemeny_recursive_22_oct_2019} one can obtain a particularly simple expression for  Kemeny's constant of a caterpillar (Proposition \ref{kemeny_caterpillars_1419_23_oct_2019}). This expression will be used to find the minimizer and the maximizer of  Kemeny's constant in the class of caterpillars having a given order and a given number of central vertices (Theorem \ref{theorem_max_min_kemeny_caterpillars_1613_18apr}). In Section \ref{sec_lower_bound_kememy}, this will lead to a sharp lower bound for Kemeny's constant of trees having fixed order and diameter.
\begin{proposition}
\label{kemeny_caterpillars_1419_23_oct_2019}
Let $C_k(\textbf{p})$ be a caterpillar of order $n\geq 2$. Then
\begin{align*}
\kappa(C_k(\textbf{p}))=\frac{1}{n-1}\left(\textbf{p}^TX\textbf{p}+2\textbf{p}^TX\textbf{e}+\frac{k^3}{3}-2nk+n^2-\frac{n}{2}+\frac{5}{3}k-\frac{1}{2}\right).
\end{align*}
\end{proposition}
\begin{proof}
We notice that $C_k(\textbf{p})$ is the concatenation of the rooted stars $(S(p_1+1),r_1),(S(p_2+1),r_2),\dots,(S(p_k+1),r_k)$, where the root $r_i$ is the central vertex of the corresponding star. From Proposition \ref{prop_kemeny_recursive_22_oct_2019} we have that
\begin{align*}
\resizebox{1\hsize}{!}{$\displaystyle
\kappa(C_k(\textbf{p}))=\frac{1}{n-1}\left(\sum_{i=1}^kp_i\kappa(S(p_i+1))
+\textbf{p}^TX(\textbf{p}+2\textbf{e})+(n-1){\bmu}^T\textbf{e}
-{\bmu}^T\textbf{p}
+\frac{k^3}{3}-kn+n+\frac{k}{6}-\frac{1}{2}\right).
$}
\end{align*}
We claim that 
\begin{equation}
\label{eqn_pik_star_1845_7nov}
p_i\kappa(S(p_i+1))=p_i^2-\frac{p_i}{2}\hspace{.8cm}(i=1,2,\dots,k).
\end{equation}
If $p_i=0$, equation \eqref{eqn_pik_star_1845_7nov} trivially holds. If $p_i\geq 1$, notice that the degree vector $\textbf{d}^{(i)}$ and the distance matrix $\Delta^{(i)}$ of the star $S(p_i+1)$ have the following simple description:
\begin{align*}
\textbf{d}^{(i)}=
\begin{bmatrix}
p_i\\
\textbf{e}
\end{bmatrix},
\hspace{.8cm}
\Delta^{(i)}=
\begin{bmatrix}
0 & \textbf{e}^T\\
\textbf{e} & 2(\textbf{e}\textbf{e}^T-I)
\end{bmatrix}
\end{align*}
where $\textbf{e}\in\amsmathbb{R}^{p_i}$ and $I$ is the identity matrix in $M_{p_i}$. Hence,
\begin{align*}
p_i\kappa(S(p_i+1))&=
p_i\frac{{\textbf{d}^{(i)}}^T\Delta^{(i)}\textbf{d}^{(i)}}{4p_i}=
\frac{1}{4}
\begin{bmatrix}
p_i & \textbf{e}^T
\end{bmatrix}
\begin{bmatrix}
0 & \textbf{e}^T\\
\textbf{e} & 2(\textbf{e}\textbf{e}^T-I)
\end{bmatrix}
\begin{bmatrix}
p_i\\
\textbf{e}
\end{bmatrix}\\
&
=
\frac{1}{4}
\begin{bmatrix}
p_i & \textbf{e}^T
\end{bmatrix}
\begin{bmatrix}
p_i\\
(3p_i-2)\textbf{e}
\end{bmatrix}
=
\frac{4p_i^2-2p_i}{4}=
p_i^2-\frac{p_i}{2}.
\end{align*}
Moreover, from Lemma \ref{lemma_moment_star} we see that $\mu_i=p_i$ for $i=1,2,\dots,k$, so that ${\bmu}=\textbf{p}$. We obtain
\begin{small}
\begin{align*}
&\kappa(C_k(\textbf{p}))
=
\frac{1}{n-1}\left(\sum_{i=1}^k\left(p_i^2-\frac{p_i}{2}\right)
+\textbf{p}^TX(\textbf{p}+2\textbf{e})+(n-1)\textbf{p}^T\textbf{e}
-\textbf{p}^T\textbf{p}
+\frac{k^3}{3}-kn+n+\frac{k}{6}-\frac{1}{2}\right)\\
&=\frac{1}{n-1}\left(\textbf{p}^T\textbf{p}-\frac{n-k}{2}+\textbf{p}^TX\textbf{p}+2\textbf{p}^TX\textbf{e}+(n-1)(n-k)
-\textbf{p}^T\textbf{p}+\frac{k^3}{3}-kn+n+\frac{k}{6}-\frac{1}{2}\right)\\
&=\frac{1}{n-1}\left(\textbf{p}^TX\textbf{p}+2\textbf{p}^TX\textbf{e}
+\frac{k^3}{3}-2nk+n^2-\frac{n}{2}+\frac{5}{3}k-\frac{1}{2}
\right)
\end{align*}
\end{small}
as desired.
\end{proof}

Let $\mc{C}_{n,k}$ denote the class of caterpillars $C_k(\textbf{p})$ having $k$ central vertices and order $n$, so that
\[
    n=|V(C_k(\textbf{p}))|=k+\sum_{i=1}^k p_i.
\]
The goal of the remaining part of this section is to find the maximum and the minimum value of  Kemeny's constant in the class $\mc{C}_{n,k}$. In particular, we will prove the following result.
\begin{theorem}
\label{theorem_max_min_kemeny_caterpillars_1613_18apr}
Let $k$ and $n$ be integers such that $2\leq k\leq n$.
\begin{enumerate}
\item[$(i)$]
The unique caterpillar in $\mathcal{C}_{n,k}$ minimizing  Kemeny's constant is $C_k((n-k)\textbf{e}_{r})$, with $r=\lceil{\frac{k}{2}}\rceil$.
\item[$(ii)$]
The unique caterpillar in $\mathcal{C}_{n,k}$ maximizing Kemeny's constant is $C_k(t_1\textbf{e}_1+t_k\textbf{e}_k)$, with $t_1=\lceil{\frac{n-k}{2}}\rceil$, $t_k=\lfloor{\frac{n-k}{2}}\rfloor$.
\end{enumerate}
\end{theorem}
\medskip
\begin{example}
 \label{ex:cat-min-max}
  Let $k=7$ and $n=15$. Then the caterpillar $C_7(\textbf{p})$ where $\textbf{p}=(0,0,0,8,0,0,0)$ minimizes $\kappa(C)$ for $C \in \mc{C}_{n,k}$, and $C_7(\textbf{p}')$ where $\textbf{p}'=(4,0,0,0,0,0,4)$ maximizes $\kappa(C)$ for the same class. 
\endproof
\end{example}
\begin{proof}[\textbf{Proof of Theorem \ref{theorem_max_min_kemeny_caterpillars_1613_18apr}}]
Consider the two functions $f,g:\amsmathbb{R}^k\rightarrow\amsmathbb{R}$ defined by
$f(\textbf{p})=\textbf{p}^TX\textbf{p}$ and $g(\textbf{p})=2\textbf{p}^TX\textbf{e}$ $(\textbf{p}\in\amsmathbb{R}^k)$. From Proposition \ref{kemeny_caterpillars_1419_23_oct_2019} we see that the problem reduces to minimizing (resp. maximizing) the function $h=f+g$ in the set $K_{n-k}=\{\textbf{p}\in \amsmathbb{Z}^k\mid \textbf{p}\geq \bzero,\; \textbf{p}^T\textbf{e}=n-k\}$. 

$(i)$ The diagonal entries of $X$ are zero, and its off-diagonal entries are strictly positive. Hence, the minimum value of $f$ in $K_{n-k}$ is attained uniquely in $(n-k)\textbf{e}_{i}$ for some $i\in\{1,2,\dots,k\}$. Moreover, a straightforward computation shows that 
\begin{equation}
\label{eq_Xei_1625_6_nov_2019}
(X\textbf{e})_i=i^2-ki-i+\frac{k^2+k}{2}\hspace{.8cm}(i=1,2,\dots,k).
\end{equation}	
If $k$ is odd, expression \eqref{eq_Xei_1625_6_nov_2019} attains its minimum only for $i=\frac{k+1}{2}=\lceil{\frac{k}{2}}\rceil$. Therefore, $g$ and $h$ are minimized uniquely for $\textbf{p}=(n-k)\textbf{e}_{\lceil{\frac{k}{2}}\rceil}$. If $k$ is even, expression \eqref{eq_Xei_1625_6_nov_2019} attains its minimum only for $i_1=\frac{k}{2}$ and $i_2=\frac{k}{2}+1$. Therefore, $h$ is minimized uniquely for $\textbf{p}_1=(n-k)\textbf{e}_{\frac{k}{2}}$ and $\textbf{p}_2=(n-k)\textbf{e}_{\frac{k}{2}+1}$. Since the two caterpillars $C_k(\textbf{p}_1)$ and $C_k(\textbf{p}_2)$ are isomorphic, and since $\frac{k}{2}=\lceil{\frac{k}{2}}\rceil$ for $k$ even, part $(i)$ of the theorem follows.

$(ii)$ Given $\textbf{v}=(v_i)\in\amsmathbb{R}^k$ and $j\in\{1,2,\dots,k\}$, we define $s_j(\textbf{v})=\sum_{i=1}^jv_i$.
Let $\textbf{p}=(p_i)\in K_{n-k}$. We define the \open right perturbation function" $R_q(\textbf{p})=h(\textbf{p}+\textbf{e}_{q+1}-\textbf{e}_q)-h(\textbf{p})$ for $q\in\{1,2,\dots,k-1\}$ with $p_q\geq 1$, and the \open left perturbation function" $L_q(\textbf{p})=h(\textbf{p}+\textbf{e}_{q-1}-\textbf{e}_q)-h(\textbf{p})$ for $q\in\{2,3,\dots,k\}$ with $p_q\geq 1$. Observe that $L_q(\textbf{p})=-R_{q-1}(\textbf{p}+\textbf{e}_{q-1}-\textbf{e}_q)$. Defining $\textbf{y}=\textbf{e}_{q+1}-\textbf{e}_q$, we see that
\begin{align*}
R_q(\textbf{p})&=h(\textbf{p}+\textbf{y})-h(\textbf{p})=f(\textbf{p}+\textbf{y})+g(\textbf{p}+\textbf{y})-f(\textbf{p})-g(\textbf{p})\\
&=f(\textbf{p})+f(\textbf{y})+2\textbf{p}^TX\textbf{y}+g(\textbf{p})
+g(\textbf{y})-f(\textbf{p})-g(\textbf{p})\\
&=f(\textbf{y})+2\textbf{p}^TX\textbf{y}
+g(\textbf{y}).
\end{align*}
Notice that $f(\textbf{y})=-2$. Additionally,
\begin{align*}
(X\textbf{y})_i=
\begin{array}{l}
\left\{\begin{aligned}
1& &\mbox{if }\; i\leq q\\
-1& &\mbox{if }\; i>q
\end{aligned}
\right.
\end{array}
\end{align*}
and, therefore, 
\begin{align*}
\textbf{p}^TX\textbf{y}&=s_q(\textbf{p})-(n-k-s_q(\textbf{p}))=2s_q(\textbf{p})-n+k,\\
g(\textbf{y})&=2\textbf{e}^TX\textbf{y}=2(q-(k-q))=4q-2k.
\end{align*}
We conclude that
\begin{align*}
R_q(\textbf{p})&=4s_q(\textbf{p})+4q-2n-2
\intertext{and, hence,}
L_q(\textbf{p})&=-R_{q-1}(\textbf{p}+\textbf{e}_{q-1}-\textbf{e}_q)\\
&=-4s_{q-1}(\textbf{p}+\textbf{e}_{q-1}-\textbf{e}_q)-4(q-1)+2n+2\\
&=-4s_{q-1}(\textbf{p})-4-4q+4+2n+2\\
&=-4s_q(\textbf{p})+4p_q-4q+2n+2\\
&=-R_q(\textbf{p})+4p_q.
\end{align*}
Suppose now that $\textbf{p}^*=(p_i^*)$ attains the maximum value of $h$ in $K_{n-k}$. If $p_q^*\geq 1$ for some $q\in\{2,3,\dots,k-1\}$, then the maximality of $\textbf{p}^*$ implies that $R_q(\textbf{p}^*)\leq 0$ and $L_q(\textbf{p}^*)\leq 0$. This means that $4p_q^*=R_q(\textbf{p}^*)+L_q(\textbf{p}^*)\leq 0$, which is a contradiction. As a consequence, there exists $t\in\{0,1,\dots,n-k\}$ such that $\textbf{p}^*=t\textbf{e}_1+(n-k-t)\textbf{e}_k$. Then, a straightforward computation yields
\begin{align*}
h(\textbf{p}^*)=t^2(2-2k)+t(2nk-2k^2-2n+2k)+nk^2-k^3-nk+k^2.
\end{align*}
Considering this expression as a quadratic polynomial in the variable $t$, we find that its maximum is attained for $t^*=\frac{n-k}{2}$, from which part $(ii)$ of the theorem readily follows.
\end{proof}

\section{A lower bound on  Kemeny's constant for trees}
\label{sec_lower_bound_kememy}
The results concerning extremal caterpillars presented in Section \ref{sec_extremal_caterpillars} provide a tool for finding a lower bound on Kemeny's constant of a tree in terms of the number of its vertices and its diameter. To do so, we shall prove a stronger version of part $(i)$ of Theorem \ref{theorem_max_min_kemeny_caterpillars_1613_18apr}. Namely, we will show that the caterpillar $C_*$ minimizing  Kemeny's constant among all the caterpillars having $n$ vertices and ${\delta}+1$ central vertices is also a minimizer in the class $\mathcal{T}_{n,{\delta}}$ of trees having $n$ vertices and diameter equal to ${\delta}$ (Theorem \ref{thm_minimization_trees_1116_31may}). $\kappa(C_*)$ will then yield a (sharp) lower bound for  Kemeny's constant in $\mathcal{T}_{n,{\delta}}$  (Corollary \ref{cor_lower_bound_kemeny_1117_31may}). 

Let $T(G)$ be the transition matrix of a nontrivial connected graph $G$ on $n$ vertices, and let $1,\lambda_2,\lambda_3,\dots,\lambda_n$ be the eigenvalues of $T(G)$. From \cite{KemenySnell} we have the following formula for Kemeny's constant of $G$:
\begin{equation*}
\kappa(G)=\sum_{j=2}^n\frac{1}{1-\lambda_j}.
\end{equation*}
We use this expression to obtain a result on Kemeny's constant for bipartite graphs (Proposition \ref{min_kemeny_bipartite_graphs}) and, as a consequence, for trees (Corollary \ref{cor_minimum_kemeny_star_1640_24_oct_2019_NEW_16_dec_2019}). The latter will be useful in the proof of Theorem \ref{thm_minimization_trees_1116_31may}.
\begin{proposition}
\label{min_kemeny_bipartite_graphs}
Let $G$ be a nontrivial connected bipartite graph on $n$ vertices. Then 
\begin{equation*}
\kappa(G)\ge n-\frac{3}{2},
\end{equation*}
with equality if and only if $G$ is a complete bipartite graph.
\end{proposition}
\begin{proof}
Let $A$ and $D$ denote the adjacency matrix and the diagonal degree matrix of $G$, respectively. We observe that 
\begin{align*}
T(G)=D^{-1}A=D^{-\frac{1}{2}}(D^{-\frac{1}{2}}AD^{-\frac{1}{2}})D^{\frac{1}{2}}.
\end{align*}
Hence, $T(G)$ is similar to the symmetric matrix $D^{-\frac{1}{2}}AD^{-\frac{1}{2}}$, so that its eigenvalues are real. Denote them by $1=\lambda_1 \ge \lambda_2 \ge \ldots \ge \lambda_n=-1.$ Since $G$ is bipartite, $\lambda_{n-j+1}=-\lambda_j,\, j=1,2,\ldots,n.$ Suppose for concreteness that exactly $q$ eigenvalues are zero, and let $\ell = \frac{n-q}{2}$. Then 
\begin{align*}
\kappa(G) &=\sum_{j=2}^n\frac{1}{1-\lambda_j} = 
\sum_{j=2}^\ell \left(\frac{1}{1-\lambda_j} + \frac{1}{1-\lambda_{n-j+1}}\right) + q + \frac{1}{2}\\
&=  q + \frac{1}{2} +\sum_{j=2}^\ell \frac{2}{1-\lambda_j^2} \ge  q + \frac{1}{2} + 2(\ell-1) =  n-\frac{3}{2}.
\end{align*}
Observe that if $q<n-2,$ then in fact  $\kappa(G)>  n-\frac{3}{2};$ it now follows that $\kappa(G)= n-\frac{3}{2}$ if and only if the transition matrix has rank $2$ -- i.e., if and only if $G$ is a complete bipartite graph.
\end{proof}
\begin{corollary}
\label{cor_minimum_kemeny_star_1640_24_oct_2019_NEW_16_dec_2019}
For any integer $n\geq 1$, the star $S(n)$ is the unique tree minimizing  Kemeny's constant value in the set of trees having $n$ vertices.
\end{corollary}
\begin{theorem}
\label{thm_minimization_trees_1116_31may}
Let $n$ and ${\delta}$ be integers such that $\mathcal{T}_{n,\delta}\neq \emptyset$. The unique tree in $\mathcal{T}_{n,{\delta}}$ minimizing  Kemeny's constant is the caterpillar $C_{\delta+1}((n-{\delta}-1)\textbf{e}_t)$, with $t=\lceil{\frac{{\delta}+1}{2}}\rceil$ and $\textbf{e}_t=(0,0,\dots,0,1,0,\dots,0)\in\amsmathbb{R}^{{\delta}+1}$.
\end{theorem}
\begin{proof}
If $\delta=0$, then $n=1$ and the result is immediate, so we assume that $\delta\geq 1$. Consider a tree $T\in\mathcal{T}_{n,{\delta}}$. If $T$ is a caterpillar, then we can write it as $T=C_{\delta+1}(\textbf{p})$ for some nonnegative integer vector $\textbf{p}=(p_1,p_2,\dots,p_{{\delta}+1})$ such that $p_1=p_{{\delta}+1}=0$, and we only need to apply part $(i)$ of Theorem \ref{theorem_max_min_kemeny_caterpillars_1613_18apr}. If $T$ is not a caterpillar, let $r_1r_2\dots r_{{\delta}+1}$ be a longest path in $T$. We can unambiguously write $T$ as the concatenation of suitable rooted trees $(T_1,r_1),(T_2,r_2),\dots,(T_{\delta+1},r_{\delta+1})$. Let $\textbf{m}=(m_i)\in\amsmathbb{R}^{\delta+1}$ be the size vector of $(T_1,r_1),(T_2,r_2),\dots,(T_{\delta+1},r_{\delta+1})$, and consider the caterpillar $\tilde{C}=C_{\delta+1}(\textbf{m})$. Notice that $\tilde{C}$ is the concatenation of the rooted stars $(S({m_1+1}),r_1),$ $(S({m_2+1}),r_2),\dots,(S({m_{\delta+1}+1}),r_{\delta+1})$, where the root $r_i$ of $S({m_i+1})$ is chosen to be the central vertex of the star. Let ${\bmu}$ be the moment vector of $(T_1,r_1),(T_2,r_2),\dots,(T_{\delta+1},r_{\delta+1})$, and let $\tilde{{\bmu}}$ be the moment vector of $(S({m_1+1}),r_1),$ $(S({m_2+1}),r_2),\dots,(S({m_{\delta+1}+1}),r_{\delta+1})$.
Using Proposition \ref{prop_kemeny_recursive_22_oct_2019}, we obtain
\begin{small}
\begin{align*}
(n-1)(\kappa(T)-\kappa(\tilde{C}))&=\sum_{i=1}^{\delta+1}m_i\kappa(T_i)
-\sum_{i=1}^{\delta+1}m_i\kappa(S(m_i+1))+(n-1)({\bmu}-\tilde{{\bmu}})^T\textbf{e}-({\bmu}-\tilde{{\bmu}})^T\textbf{m}\\
&=\sum_{i=1}^{\delta+1}m_i(\kappa(T_i)-\kappa(S({m_i+1})))+({\bmu}-\tilde{{\bmu}})^T
((n-1)\textbf{e}-\textbf{m}).
\end{align*}
\end{small}
Corollary \ref{cor_minimum_kemeny_star_1640_24_oct_2019_NEW_16_dec_2019} shows that $\kappa(T_i)\geq\kappa(S({m_i+1}))$ for each $i$; since $T$ is not a caterpillar, moreover, there exists an index $j$ such that $m_{j}\geq 1$ and the previous inequality is strict. We thus obtain
\begin{align*}
(n-1)(\kappa(T)-\kappa(\tilde{C}))>({\bmu}-\tilde{{\bmu}})^T
((n-1)\textbf{e}-\textbf{m}).
\end{align*}
The vector $((n-1)\textbf{e}-\textbf{m})$ is entrywise nonnegative. Additionally, Lemma \ref{lemma_moment_star} shows that $\bmu\geq \tilde\bmu$, so that ${\bmu}-\tilde{{\bmu}}$ is also entrywise nonnegative. We deduce that
\begin{align*}
(n-1)(\kappa(T)-\kappa(\tilde{C}))>0
\end{align*}
and, hence, $\kappa(T)>\kappa(\tilde{C})$. Since $\tilde{C}\in\mathcal{T}_{n,\delta}$, $T$ does not minimize  Kemeny's constant in $\mathcal{T}_{n,\delta}$.
\end{proof}
\begin{corollary}
\label{cor_lower_bound_kemeny_1117_31may}
Let $T$ be a nontrivial tree of order $n$ and diameter ${\delta}$. Then
\begin{itemize}
\item
if ${\delta}$ is odd, \hspace{.15cm}
$
\displaystyle
\kappa(T)\geq\frac{1}{n-1}\left(\frac{n{\delta}^2}{2}-\frac{{\delta}^3}{6}+n^2-n{\delta}-\frac{{\delta}^2}{2}-2n+\frac{7}{6}{\delta}+1 \right);
$
\item
if ${\delta}$ is even, \hspace{.15cm}
$
\displaystyle
\kappa(T)\geq \frac{1}{n-1}\left(\frac{n{\delta}^2}{2}-\frac{{\delta}^3}{6}+n^2-n{\delta} -\frac{{\delta}^2}{2}-\frac{5}{2}n+\frac{5}{3}{\delta}+\frac{3}{2} \right).
$
\end{itemize}
Additionally, in both cases, $T$ satisfies the bound with equality if and only if $T$ is the caterpillar $C_{\delta+1}((n-{\delta}-1)\textbf{e}_t)$, with $t=\lceil{\frac{{\delta}+1}{2}}\rceil$.
\end{corollary}
\begin{proof}
The result is obtained by applying Theorem \ref{thm_minimization_trees_1116_31may} and by computing  Kemeny's constant of the caterpillar $C_{\delta+1}((n-{\delta}-1)\textbf{e}_t)$ using the formula given in Proposition \ref{kemeny_caterpillars_1419_23_oct_2019}.  
\end{proof}
\section{An upper bound on  Kemeny's constant for trees}
\label{sec_upper_bound}
One can check by inspection that the maximizer of  Kemeny's constant in $\mc{T}_{n,\delta}$ -- i.e., the set of trees having order $n$ and diameter $\delta$ -- is not a caterpillar in general. Hence, contrary to what happens for the lower bound (see Section \ref{sec_lower_bound_kememy}), we do not obtain an upper bound for Kemeny's constant in $\mc{T}_{n,\delta}$ from Proposition \ref{kemeny_caterpillars_1419_23_oct_2019}.

\begin{example}
	Note that any tree on $n \ge 4$ vertices with diameter $3$ is necessarily a caterpillar. Hence, by Theorem \ref{theorem_max_min_kemeny_caterpillars_1613_18apr}, the tree on $n$ vertices with diameter $3$ that maximizes Kemeny's constant is 
	$C_2(\lceil{\frac{n-2}{2}}\rceil\textbf{e}_1+\lfloor{\frac{n-2}{2}}\rfloor\textbf{e}_2)$; it is straightforward to determine that this maximum value is $\frac{3}{2}n-3+\frac{1}{2n-2}$ if $n$ is even, $\frac{3}{2}n-3$ if $n$ is odd.
	
	For trees of diameter $4,$ the maximizer of Kemeny's constant may no longer be a caterpillar. For instance,  consider trees on $13$ vertices with diameter $4$.  The value of Kemeny's constant for the caterpillar $C_3(5\textbf{e}_1 + 5\textbf{e}_3)$ 
 is equal to $\frac{43}{2}.$ In constrast, consider the tree $T$ on $13$ vertices formed from three copies of $S(4)$ by making each of their central vertices adjacent to a common $13$--th vertex. Then $T$ has diameter $4,$ but $\kappa(T) = \frac{47}{2} > \kappa(C_3(5\textbf{e}_1 + 5\textbf{e}_3)).$ 
\end{example}

Nevertheless, an upper bound can be found by means of an iterative use of Corollary \ref{cor_kemeny_pendent_vertex}.
\begin{theorem}
\label{thm_upper_bound_kemeny_29_oct_NEW}
Let $T$ be a tree of order $n$ and diameter $\delta$. Then
\begin{align*}
\kappa(T)\leq n\delta-\frac{\delta^2}{2}.
\end{align*}
\end{theorem}
\begin{remark}
A different upper bound for  Kemeny's constant of a tree $T$ of order $n$ and diameter $\delta$ can be derived from expression \eqref{eq_kemeny_constant_tree_1723_05_nov} by looking at the norms of the degree vector $\textbf{d}$ and of the distance matrix $\Delta$ separately. For example, one easily obtains
\begin{align*}
\kappa(T)=\frac{\textbf{d}^T\Delta\textbf{d}}{4(n-1)}\leq \frac{\rho(\Delta)\|\textbf{d}\|_2^2}{4(n-1)}\leq \frac{\rho(\Delta)\|\textbf{d}\|_1^2}{4(n-1)}=\frac{\rho(\Delta)(2(n-1))^2}{4(n-1)}=\rho(\Delta)(n-1)
\end{align*}
where $\rho(\Delta)$ is the spectral radius of $\Delta$. Using for example \cite[Theorem 8.1.22.]{HJ}, we have that
\begin{align*}
\rho(\Delta)\geq \min_{1\leq i\leq n}\sum_{j=1}^n\Delta_{ij}\geq n-1.
\end{align*}
Hence, this method would provide an upper bound asymptotically larger than or equal to $n^2$, thus worse than the one given in Theorem \ref{thm_upper_bound_kemeny_29_oct_NEW}.
\end{remark}
Before proving Theorem \ref{thm_upper_bound_kemeny_29_oct_NEW}, we present a sharp upper bound for the moment of a rooted tree. Whoever has ever played on a seesaw can confirm that to maximize the moment of the gravity force one needs to push all the weight as far from the fulcrum as possible. Proposition \ref{prop_upper_bound_moment_of_inertia_1546_1_nov} shows that the \textit{seesaw principle} applies to the moment of a rooted tree too. Given integers $x\geq 1$ and $y\geq 0$, we define the \textit{rooted broom} $B(x,y)$ as the rooted tree obtained by attaching $y$ pendent vertices to an endpoint of a path of $x$ vertices, and by letting the other endpoint be the root (if $x=1$, $B(x,y)$ is the star $S(y+1)$, with the central vertex as root; see Figure \ref{fig_rooted_broom}). We also let $B(0,1)$ be the trivial tree, with the unique vertex as root. 
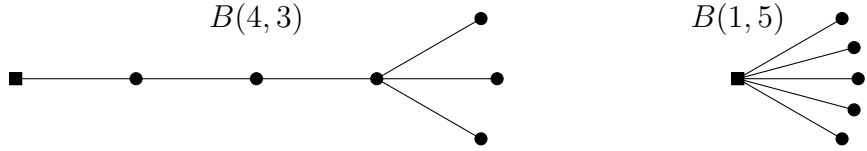
\begin{figure}[h]
\centering
\begin{tikzpicture}[scale=.8]
%
\draw  (4,1) node{$B(4,3)$};
\draw [fill,draw=black] (0-.1,-.1) rectangle ++(.2,.2);
\draw[fill]  (2,0) circle (0.1cm);
\draw[fill]  (4,0) circle (0.1cm);
\draw[fill]  (6,0) circle (0.1cm);
\draw[fill]  (6+1.7321,1) circle (0.1cm);
\draw[fill]  (6+1.7321,-1) circle (0.1cm);
\draw[fill]  (8,0) circle (0.1cm);
%
\draw[black] (0,0) -- (2,0) -- (4,0) -- (6,0);
\draw[black] (6,0) -- (6+1.7321,1);
\draw[black] (6,0) -- (6+1.7321,-1);
\draw[black] (6,0) -- (8,0);
%
%
\begin{scope}[shift={(12,0)}]
%
\draw  (0,1) node{$B(1,5)$};
\draw [fill,draw=black] (0-.1,-.1) rectangle ++(.2,.2);
\draw[fill]  (0+1.7321,1) circle (0.1cm);
\draw[fill]  (0+1.7321,-1) circle (0.1cm);
\draw[fill]  (1.9319,0.5176) circle (0.1cm);
\draw[fill]  (1.9319,-0.5176) circle (0.1cm);
\draw[fill]  (2,0) circle (0.1cm);
%
\draw[black] (0,0) -- (0+1.7321,1);
\draw[black] (0,0) -- (0+1.7321,-1);
\draw[black] (0,0) -- (1.9319,0.5176);
\draw[black] (0,0) -- (1.9319,-0.5176);
\draw[black] (0,0) -- (2,0);
\end{scope}
\end{tikzpicture}
\caption{\footnotesize{Two examples of rooted broom.}}
\label{fig_rooted_broom}
\end{figure}
\begin{proposition}
\label{prop_upper_bound_moment_of_inertia_1546_1_nov}
Let $T$ be a rooted tree of order $n$ and diameter $\delta$. Then its moment $\mu(T)$ satisfies
\begin{align*}
\mu(T)\leq 2n\delta-\delta^2-n-\delta+1,
\end{align*}
with equality if and only if $T=B(\delta,n-\delta)$. 
\end{proposition}
\begin{proof}
For $i=0,1,\dots$, let $S_i$ be the set of vertices in $T$ having distance $i$ from the root $r$ of $T$, and let $s_i=|S_i|$. Define $t=\max\{i\mid s_i\geq 1\}$, and observe that $s_i\geq 1$ for $0\leq i \leq t$. Notice, also, that $0\leq t\leq \delta$. Using the definition \eqref{def_moment_of_inertia_1501_01_nov} of moment, we see that
\begin{align*}
\mu(T)&=\sum_{v\in V(T)}\dist(v,r)\deg(v)=\sum_{i=1}^ti\sum_{v\in S_i}\deg(v).
\end{align*}
Each vertex in $S_i$ ($i=1,2,\dots,t$) is adjacent to exactly one vertex in $S_{i-1}$. Moreover, no edge links a vertex in $S_i$ with a vertex in $S_j$ unless $j=i+1$ or $j=i-1$. This implies that $\sum_{v\in S_i}\deg(v)=|S_i|+|S_{i+1}|=s_i+s_{i+1}$ for $i=1,2,\dots,t$. We obtain
\begin{align*}
\mu(T)&=\sum_{i=1}^ti(s_i+s_{i+1})=\sum_{i=1}^tis_i+\sum_{i=2}^{t+1}(i-1)s_i=
\sum_{i=1}^tis_i+\sum_{i=1}^{t}(i-1)s_i\\
&=2\sum_{i=1}^tis_i-n+1
=2\sum_{i=1}^ti+2\sum_{i=1}^ti(s_i-1)-n+1\\
&=t^2+t-n+1+2\sum_{i=1}^ti(s_i-1)\leq t^2+t-n+1+2\sum_{i=1}^tt(s_i-1)\\
&=t^2+t-n+1+2t(n-1)-2t^2=2nt-t^2-n-t+1.
\end{align*}
Equality holds if and only if $(t-i)(s_i-1)=0$ $\forall i=1,2,\dots,t$ or, equivalently, if and only if $s_i=1$ $\forall i=1,2,\dots,t-1$. This occurs precisely when $T=B(t,n-t)$. Since
\begin{align*}
\frac{\partial}{\partial x}(2nx-x^2-n-x+1)=2n-2x-1>0 \hspace{.4cm}\mbox{for} \hspace{.4cm}t\leq x\leq \delta,
\end{align*}
we have that
\begin{align*}
2nt-t^2-n-t+1\leq 2n\delta-\delta^2-n-\delta+1,
\end{align*}
with equality if and only if $t=\delta$. We conclude that 
\begin{align*}
\mu(T)\leq 2n\delta-\delta^2-n-\delta+1,
\end{align*}
with equality if and only if $T=B(\delta,n-\delta)$.
\end{proof}
\begin{proof}[\textbf{Proof of Theorem \ref{thm_upper_bound_kemeny_29_oct_NEW}}]
We use induction on the order $n$ of $T$. If $n=1$, then $\kappa(T)=0$ by definition and the theorem holds. Let now $T$ be a tree of order $n\geq 2$ and diameter $\delta$, and suppose the theorem is true for trees of order up to $n-1$. Let $P$ be a longest path in $T$, and let $v$ be one of the two endpoints of $P$. Observe that $v$ is a pendent vertex in $T$, and let $w$ be its neighbor. Using Corollary \ref{cor_kemeny_pendent_vertex} we obtain
\begin{align}
\label{eqn_1721_22_nov_2019}
\kappa(T)=\frac{1}{n-1}\left((n-2)\kappa(T-v)+\mu(T-v,w)+n-\frac{3}{2}\right).
\end{align}
If $\diam(T-v)=\delta$, by Proposition \ref{prop_upper_bound_moment_of_inertia_1546_1_nov}
\begin{align*}
\mu(T-v,w)\leq 2(n-1)\delta-\delta^2-(n-1)-\delta+1=2n\delta-\delta^2-n-3\delta+2.
\end{align*}
If $\diam(T-v)=\delta-1$, by Proposition \ref{prop_upper_bound_moment_of_inertia_1546_1_nov}
\begin{align*}
\mu(T-v,w)&\leq 2(n-1)(\delta-1)-(\delta-1)^2-(n-1)-(\delta-1)+1\\
&=2n\delta-\delta^2-3n-\delta+4\\
&=(2n\delta-\delta^2-n-3\delta+2)-2n+2\delta+2\\
&\leq 2n\delta-\delta^2-n-3\delta+2.
\end{align*}
In either case, then, we have that $\mu(T-v,w)\leq 2n\delta-\delta^2-n-3\delta+2$.
Moreover, applying the inductive hypothesis to $T-v$ yields
\begin{align*}
\kappa(T-v)&\leq (n-1)\delta-\frac{\delta^2}{2}=n\delta-\frac{\delta^2}{2}-\delta
\intertext{if $\diam(T-v)=\delta$, and}
\kappa(T-v)&\leq (n-1)(\delta-1)-\frac{(\delta-1)^2}{2}=n\delta-\frac{\delta^2}{2}-n+\frac{1}{2}\\
&=\left(n\delta-\frac{\delta^2}{2}-\delta\right)+\delta-n+\frac{1}{2}\\
&\leq n\delta-\frac{\delta^2}{2}-\delta
\end{align*}
if $\diam(T-v)=\delta-1$. Hence, in either case $\kappa(T-v)\leq n\delta-\frac{\delta^2}{2}-\delta$. Substituting this into \eqref{eqn_1721_22_nov_2019}, we obtain
\begin{align*}
\kappa(T)&\leq
\frac{1}{n-1}\left((n-2)\left(n\delta-\frac{\delta^2}{2}-\delta\right)+2n\delta-\delta^2-n-3\delta+2+n-\frac{3}{2}\right)\\
&=\frac{n^2\delta-\frac{n\delta^2}{2}-n\delta-\delta+\frac{1}{2}}{n-1}\\
&=\frac{(n-1)(n\delta-\frac{\delta^2}{2})-\frac{\delta^2}{2}-\delta+\frac{1}{2}}{n-1}\\
&\leq
\frac{(n-1)(n\delta-\frac{\delta^2}{2})}{n-1}\\
&=n\delta-\frac{\delta^2}{2},
\end{align*}
thus validating the inductive step and concluding the proof of the theorem.
\end{proof}
Putting Corollary \ref{cor_lower_bound_kemeny_1117_31may} and Theorem \ref{thm_upper_bound_kemeny_29_oct_NEW} together, we obtain the following main result.
\begin{theorem}
Let $T$ be a tree of order $n$ and diameter $\delta$. Then
\begin{equation*}
n+\frac{\delta^2}{3}-\delta-1\;\leq\; \kappa(T)\;\leq\; n\delta-\frac{\delta^2}{2}.
\end{equation*}
\end{theorem}
\begin{proof}
The second inequality is Theorem \ref{thm_upper_bound_kemeny_29_oct_NEW}. 
We easily check that the first inequality is satisfied for the trivial tree (whose Kemeny's constant value is $0$) and for the unique tree of order $2$ (whose Kemeny's constant value is $1/2$), so we assume that $n\geq 3$. We apply Corollary \ref{cor_lower_bound_kemeny_1117_31may} as follows. If $\delta$ is odd, 
\begin{align*}
(n-1)\kappa(T)&\geq
\frac{n{\delta}^2}{2}-\frac{{\delta}^3}{6}+n^2-n{\delta}-\frac{{\delta}^2}{2}-2n+\frac{7}{6}{\delta}+1\\
&=(n-1)\left(n+\frac{\delta^2}{3}-\delta-1\right)
+\frac{1}{6}(n\delta^2-\delta^3-\delta^2+\delta)\\
&\geq (n-1)\left(n+\frac{\delta^2}{3}-\delta-1\right),
\end{align*}
where the last inequality is due to the fact that 
\begin{align*}
n\delta^2-\delta^3-\delta^2+\delta\geq (\delta+1)\delta^2-\delta^3-\delta^2+\delta=\delta\geq 0.
\end{align*}
If $\delta$ is even,
\begin{align*}
(n-1)\kappa(T)&\geq
\frac{n{\delta}^2}{2}-\frac{{\delta}^3}{6}+n^2-n{\delta} -\frac{{\delta}^2}{2}-\frac{5}{2}n+\frac{5}{3}{\delta}+\frac{3}{2}\\
&=(n-1)\left(n+\frac{\delta^2}{3}-\delta-1\right)
+\frac{1}{6}(n\delta^2-\delta^3-\delta^2-3n+4\delta+3)\\
&\geq (n-1)\left(n+\frac{\delta^2}{3}-\delta-1\right),
\end{align*}
where the last inequality is due to the fact that the polynomial $p(n,\delta)=n\delta^2-\delta^3-\delta^2-3n+4\delta+3$ is nonnegative when $n\geq \delta+1$ and $\delta\geq 2$ (since $p(\delta+1,\delta)=\delta\geq 0$ and $\frac{\partial p}{\partial n}=\delta^2-3\geq 0$). The conclusion follows, since a tree of order $n\geq 3$ has diameter $\delta\geq 2$.
\end{proof}
\section{Asymptotic analysis of the upper bound}
\label{sec_asymptotic_analysis_upper_bound}
The lower bound for Kemeny's constant presented in Corollary \ref{cor_lower_bound_kemeny_1117_31may} is sharp: given two integers $n$ and $\delta$ such that $n\geq 2$ and $\mathcal{T}_{n,\delta}\neq \emptyset$, there exists a tree of order $n$ and diameter $\delta$ whose Kemeny's constant equals the bound -- namely, the caterpillar $C_{\delta+1}((n-{\delta}-1)\textbf{e}_t)$, with $t=\lceil{\frac{{\delta}+1}{2}}\rceil$. The same does not hold for the upper bound 
\begin{equation}
\label{eqn_upper_bound_recall_30_12_2019}
\kappa(T)\leq n\delta-\frac{\delta^2}{2}
\end{equation}
given in Theorem \ref{thm_upper_bound_kemeny_29_oct_NEW}. As an example, for the unique tree $U$ having order $n=2$ and diameter $\delta=1$, we get $\kappa(U)=1/2<3/2=n\delta-\delta^2/2$. Nevertheless, it can be shown that the upper bound \eqref{eqn_upper_bound_recall_30_12_2019} is \textit{asymptotically} sharp, in the sense described in the following theorem. 
\begin{theorem}
\label{thm_asymptotic_sharpness_upper_bound_30_12_2019}
There exists a sequence $(T_1,T_2,\dots)$ of trees such that, letting $n_i$ be the order of $T_i$ and $\delta_i$ be its diameter for $i=1,2,\dots$, 
\begin{equation*}
\begin{array}{l}
\displaystyle\lim_{i\rightarrow\infty}n_i=\infty\hspace{2cm}\mbox{and}\\
\displaystyle\lim_{i\rightarrow\infty}\frac{\kappa(T_i)}{n_i\delta_i-\delta_i^2/2}=1.
\end{array}
\end{equation*}
\end{theorem}
The remaining part of this section is dedicated to proving Theorem \ref{thm_asymptotic_sharpness_upper_bound_30_12_2019}. A first natural candidate to look at in order to check the asymptotic sharpness of the bound \eqref{eqn_upper_bound_recall_30_12_2019} is the maximizer of Kemeny's constant within the class of caterpillars having $n$ vertices and $k$ central vertices. From part $(ii)$ of Theorem \ref{theorem_max_min_kemeny_caterpillars_1613_18apr}, we know that its expression is $\tilde C=C_k(t_1\textbf{e}_1+t_k\textbf{e}_k)$, with $t_1=\lceil{\frac{n-k}{2}}\rceil$ and $t_k=\lfloor{\frac{n-k}{2}}\rfloor$. Assume, for simplicity, that $n-k$ is even and at least $2$, so that $\delta=k+1$. Using Proposition \ref{kemeny_caterpillars_1419_23_oct_2019}, we find that
\begin{align*}
\kappa(\tilde C)&=\frac{1}{n-1}\left(
\frac{n^2\delta}{2}-\frac{\delta^3}{6}-2n\delta+\delta^2+\frac{3}{2}n+\frac{\delta}{6}-\frac{3}{2}
\right).
\end{align*}
If the order $n$ is much larger than the diameter $\delta$, we see that $\kappa(\tilde{C})\sim \frac{n\delta}{2}$. As a consequence, if $n\gg \delta$, then $\kappa(\tilde{C})$ is -- asymptotically -- one half of the bound \eqref{eqn_upper_bound_recall_30_12_2019}. 

Let us take a closer look at the structure of $\tilde C$ in order to figure out how to increase its Kemeny's constant. $\tilde{C}$ consists of two star-like clusters connected by a path. Its relatively high Kemeny's constant value is due to the fact that escaping from one cluster to reach the other is quite laborious: a random walker starting in one of the clusters will most likely remain trapped for a long time before managing to find the central path and, eventually, reach the other cluster. On the other hand, traveling from one vertex to another in the same cluster is faster. A good strategy to further increase the expected travel-time (and  Kemeny's constant) is to add more clusters. By doing so, it will be less likely that two randomly chosen vertices belong to the same cluster. In this way, we end up with a particular tree called \textit{broom-star}: given integers $t\geq 2$, $q\geq 2$ and $p\geq 1$, the broom-star $BS(t,q,p)$ is the tree obtained by considering $t$ disjoint copies of the rooted broom $B(q,p)$, and by identifying their $t$ roots in a unique vertex which we call the \textit{center} of the broom-star (Figure \ref{fig_broom_star}). We shall refer to the $t$ paths of $q$ vertices hanging from the center as to the \textit{arms} of the broom-star. Observe that the diameter of $BS(t,q,p)$ is twice the length of the arms: $\delta=2q$. Moreover, the order of $BS(t,q,p)$ is $n=tp+tq-t+1$. 
\begin{figure}[h]
\centering
\begin{tikzpicture}[scale=.4]
\draw [] (0,0) circle (0.25cm);
\draw[fill]  (2,0) circle (0.1cm);
\draw[fill]  (4,0) circle (0.1cm);
\draw[fill]  (6,0) circle (0.1cm);
\draw[fill]  (6+1.7321,1) circle (0.1cm);
\draw[fill]  (6+1.7321,-1) circle (0.1cm);
\draw[fill]  (8,0) circle (0.1cm);
%
\draw[black] (.25,0) -- (2,0) -- (4,0) -- (6,0);
\draw[black] (6,0) -- (6+1.7321,1);
\draw[black] (6,0) -- (6+1.7321,-1);
\draw[black] (6,0) -- (8,0);
%
\begin{scope}[rotate=60]
\draw[fill]  (2,0) circle (0.1cm);
\draw[fill]  (4,0) circle (0.1cm);
\draw[fill]  (6,0) circle (0.1cm);
\draw[fill]  (6+1.7321,1) circle (0.1cm);
\draw[fill]  (6+1.7321,-1) circle (0.1cm);
\draw[fill]  (8,0) circle (0.1cm);
%
\draw[black] (.25,0) -- (2,0) -- (4,0) -- (6,0);
\draw[black] (6,0) -- (6+1.7321,1);
\draw[black] (6,0) -- (6+1.7321,-1);
\draw[black] (6,0) -- (8,0);
\end{scope}
%
\begin{scope}[rotate=120]
\draw[fill]  (2,0) circle (0.1cm);
\draw[fill]  (4,0) circle (0.1cm);
\draw[fill]  (6,0) circle (0.1cm);
\draw[fill]  (6+1.7321,1) circle (0.1cm);
\draw[fill]  (6+1.7321,-1) circle (0.1cm);
\draw[fill]  (8,0) circle (0.1cm);
%
\draw[black] (.25,0) -- (2,0) -- (4,0) -- (6,0);
\draw[black] (6,0) -- (6+1.7321,1);
\draw[black] (6,0) -- (6+1.7321,-1);
\draw[black] (6,0) -- (8,0);
\end{scope}
%
\begin{scope}[rotate=180]
\draw[fill]  (2,0) circle (0.1cm);
\draw[fill]  (4,0) circle (0.1cm);
\draw[fill]  (6,0) circle (0.1cm);
\draw[fill]  (6+1.7321,1) circle (0.1cm);
\draw[fill]  (6+1.7321,-1) circle (0.1cm);
\draw[fill]  (8,0) circle (0.1cm);
%
\draw[black] (.25,0) -- (2,0) -- (4,0) -- (6,0);
\draw[black] (6,0) -- (6+1.7321,1);
\draw[black] (6,0) -- (6+1.7321,-1);
\draw[black] (6,0) -- (8,0);
\end{scope}
%
\begin{scope}[rotate=240]
\draw[fill]  (2,0) circle (0.1cm);
\draw[fill]  (4,0) circle (0.1cm);
\draw[fill]  (6,0) circle (0.1cm);
\draw[fill]  (6+1.7321,1) circle (0.1cm);
\draw[fill]  (6+1.7321,-1) circle (0.1cm);
\draw[fill]  (8,0) circle (0.1cm);
%
\draw[black] (.25,0) -- (2,0) -- (4,0) -- (6,0);
\draw[black] (6,0) -- (6+1.7321,1);
\draw[black] (6,0) -- (6+1.7321,-1);
\draw[black] (6,0) -- (8,0);
\end{scope}
%
\begin{scope}[rotate=300]
\draw[fill]  (2,0) circle (0.1cm);
\draw[fill]  (4,0) circle (0.1cm);
\draw[fill]  (6,0) circle (0.1cm);
\draw[fill]  (6+1.7321,1) circle (0.1cm);
\draw[fill]  (6+1.7321,-1) circle (0.1cm);
\draw[fill]  (8,0) circle (0.1cm);
%
\draw[black] (.25,0) -- (2,0) -- (4,0) -- (6,0);
\draw[black] (6,0) -- (6+1.7321,1);
\draw[black] (6,0) -- (6+1.7321,-1);
\draw[black] (6,0) -- (8,0);
\end{scope}
\end{tikzpicture}
\caption{\footnotesize{The broom-star $BS(6,4,3)$. The center is the vertex in white.}}
\label{fig_broom_star}
\end{figure}
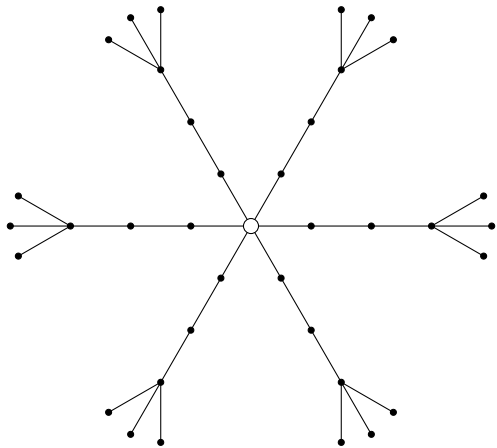

Attempting to find an explicit expression for  Kemeny's constant of a broom-star by using the general formula \eqref{eq_kemeny_constant_tree_1723_05_nov} turns out to be a laborious task. Instead, our strategy will consist of the following three steps:
\begin{enumerate}
\item[$(I)$]
exhibit a recursive formula that the numbers $\kappa(BS(t,q,p))$ satisfy by virtue of Proposition \ref{prop_kemeny_recursive_22_oct_2019};
\item[$(II)$]
guess an expression for $\kappa(BS(t,q,p))$; 
\item[$(III)$]
verify that the guess of step $(II)$ satisfies the recursive formula of step $(I)$.
\end{enumerate}
This will finally lead to the explicit expression of Proposition \ref{prop_kemeny_constant_broom_star} which, in turn, will allow to prove Theorem \ref{thm_asymptotic_sharpness_upper_bound_30_12_2019}. Step $(I)$ is performed in the next proposition, where Kemeny's constant of a broom-star is expressed in terms of Kemeny's constant of a smaller broom-star, having one less arm. For simplicity, we define $\kappa_{t,q,p}=\kappa(BS(t,q,p))$. 
\begin{proposition}
\label{prop_recursive_rule_kemeny_broom_stars}
Let $q\geq 2$ and $p\geq 1$ be integers. Then
\begin{itemize}
\item
$\displaystyle \kappa_{2,q,p}=\frac{1}{p+q-1}\left(
2p^2q + 4pq^2 + \frac{4}{3}q^3 - 6pq - 4q^2 + \frac{3}{2}p + \frac{25}{6}q - \frac{3}{2}
\right)$;
\item
$\displaystyle \kappa_{t+1,q,p}=\frac{A\kappa_{t,q,p}+B}{C}$ \hspace{.3cm} for $t=2,3,\dots$,       where
\begin{align*}
A&=pt+qt-t,\\
B&=\frac{q^3}{3}-pq+4tp+6tq-2{p}^{2}t-6{q}^{2}t+2{q}^{3}t+p{q}^{
2}-\frac{1}{2}-\frac{p}{2}-2t+{p}^{2}-{q}^{2}\\
&+6p{q}^{2}t-10pqt+4{p}^{2}qt+\frac{7}{6}q,
\\
C&=pt+qt+p+q-t-1.
\end{align*}
\end{itemize}
\end{proposition}
\begin{proof}
Observe that the broom-star $BS(2,q,p)$ coincides with the caterpillar $C_{2q-1}(\textbf{p})$ where $\textbf{p}=p\,\textbf{e}_1+p\,\textbf{e}_{2q-1}\in\amsmathbb{N}^{2q-1}$. As a consequence, we obtain the expression for $\kappa_{2,q,p}$ by computing  Kemeny's constant of $C_{2q-1}(\textbf{p})$ via Proposition \ref{kemeny_caterpillars_1419_23_oct_2019}.

Suppose that $t\geq 2$. We can view the broom-star $BS(t+1,q,p)$ as the concatenation of rooted trees as follows. Let $T_1$ be the star $S(p+1)$ rooted in the central vertex, $T_2,T_3,\dots,T_{q-1}$ be copies of the trivial tree rooted in the unique vertex, and $T_q$ be the broom-star $BS(t,q,p)$ rooted in the center. Then, $BS(t+1,q,p)=\conc(T_1,T_2,\dots,T_q)$. We can compute $\kappa_{t+1,q,p}$ via Proposition \ref{prop_kemeny_recursive_22_oct_2019}:
\begin{equation}
\label{eqn_rec_formula_broom_star_conc_1801_14_nov}
\resizebox{1\hsize}{!}{$\displaystyle
\kappa_{t+1,q,p}
=
\frac{1}{n-1}\left(\sum_{i=1}^qm_i\kappa(T_i)
+\textbf{m}^TX(\textbf{m}+2\textbf{e})+(n-1){\bmu}^T\textbf{e}
-{\bmu}^T\textbf{m}
+\frac{q^3}{3}-qn+n+\frac{q}{6}-\frac{1}{2}\right)
$}
\end{equation}
where $\textbf{m}=(m_i)$ and ${\bmu}=(\mu_i)$ are the size vector and the moment vector of the concatenation, respectively, and $n$ is the order of $BS(t+1,q,p)$. Observe that 
\begin{align*}
&n=(t+1)(p+q)-t\;;\\
&\textbf{m}=\begin{bmatrix}
p & 0 & \dots & 0 & tp+tq-t
\end{bmatrix}^T;\\
&{\bmu}=\begin{bmatrix}
p & 0 & \dots & 0 & t\cdot \mu(B(q,p))
\end{bmatrix}^T
=\begin{bmatrix}
p & 0 & \dots & 0 & t(2pq+q^2-p-2q+1))
\end{bmatrix}^T;\\
&\sum_{i=1}^qm_i\kappa(T_i)=p\,\kappa(S(p+1))+(tp+tq-t)\,\kappa(BS(t,q,p))=p^2-\frac{p}{2}+(tp+tq-t)\kappa_{t,q,p}
\end{align*}
(in the last line we have used the same argument as for \eqref{eqn_pik_star_1845_7nov}). Substituting this into \eqref{eqn_rec_formula_broom_star_conc_1801_14_nov} yields the desired expression for $\kappa_{t+1,q,p}$.
\end{proof}
We have thus concluded step $(I)$. Step $(II)$ was performed via MATLAB, by applying multivariate regression on a sample of $100$ randomly generated broom-stars, and it produced the candidate expression \eqref{expression_kemeny_broom_stars_1912_14_nov}. To show that this candidate is correct (step $(III)$), we only need to check that it satisfies the recursive relation of Proposition \ref{prop_recursive_rule_kemeny_broom_stars}. 
\begin{proposition}
\label{prop_kemeny_constant_broom_star}
Let $t\geq 2$, $q\geq 2$ and $p\geq 1$ be integers. Kemeny's constant of the broom-star $BS(t,q,p)$ is
\begin{equation}
\label{expression_kemeny_broom_stars_1912_14_nov}
\kappa_{t,q,p}=
2pqt+{q}^{2}t-\frac{2}{3}{p}^{2}-\frac{4}{3}pq-tp-\frac{2}{3}{q}^{2}-2tq+\frac{4}{3}p+\frac{4}{3}q+t-\frac{1}{2}+\frac{2}{3}p\frac{{p}^{2}-1}{p+q-1}.
\end{equation}
\end{proposition}
\begin{proof}
For any fixed $q^*\geq 2$ and $p^*\geq 1$, the sequence $\kappa_{2,q^*,p^*},\kappa_{3,q^*,p^*},\dots$ defined by \eqref{expression_kemeny_broom_stars_1912_14_nov} satisfies the requirements of Proposition \ref{prop_recursive_rule_kemeny_broom_stars}. The result follows since, clearly, those requirements cannot be satisfied by two distinct sequences.
\end{proof}
\begin{proof}[\textbf{Proof of Theorem \ref{thm_asymptotic_sharpness_upper_bound_30_12_2019}}]
Let $T_1$ be a nontrivial tree and, for $i=2,3,\dots$, let $T_i=BS(i,i,i^2)$. Observe that, if $i\geq 2$, the diameter of $T_i$ is $\delta_i=2i$ and its order is $n_i=i^3+i^2-i+1$. In particular,
\begin{equation*}
\lim_{i\rightarrow\infty}n_i=\infty.
\end{equation*}
Moreover, using Proposition \ref{prop_kemeny_constant_broom_star}, we find that
\begin{align*}
\lim_{i\rightarrow\infty}\frac{\kappa(T_i)}{n_i\delta_i-\delta_i^2/2}
&=
\lim_{i\rightarrow\infty}\frac{\frac{1}{3}\left(4i^4-4i^3-4i^2+7i-\frac{3}{2}+2\frac{i^2(i^4-1)}{i^2+i-1}\right)}
{2i^4+2i^3-4i^2+2i}=1.
\end{align*}
Therefore, the sequence $(T_1,T_2,\dots)$ satisfies the requirements in the statement of the theorem. 
\end{proof}

\bigskip
\noindent\textbf{Acknowledgements}\\

S.K.'s research is supported in part by the Natural Sciences and Engineering Research Council of Canada under grant number RGPIN--2019--05408.


\bigskip

\begin{thebibliography}{99}
%
%
\bibitem{Breen} J. Breen, Markov chains under combinatorial constraints: analysis and synthesis, PhD Thesis, University of Manitoba (2018).
%
\bibitem{Catral} M. Catral, S. Kirkland, M. Neumann, N.-S. Sze, The Kemeny constant for finite homogeneous ergodic Markov chains, Journal of Scientific Computing 45(1--3) (2010): 151--166.
%
\bibitem{Ciardo_braess} L. Ciardo, The Braess' paradox for pendent twins, Linear Algebra and its Applications 590 (2020): 304--316.
%
\bibitem{HJ} R.A. Horn, C.R. Johnson, \textit{Matrix theory}, Cambridge University Press (1986).
%
\bibitem{HuKirkland} Y. Hu, S. Kirkland, Complete multipartite graphs and Braess edges, Linear Algebra and its Applications 579 (2019): 284--301.
%
\bibitem{Hunter} J.J. Hunter, The role of Kemeny's constant in properties of Markov chains, Communications in Statistics - Theory and Methods 43(7) (2014): 1309--1321.
%
\bibitem{KemenySnell} J. Kemeny, J. Snell, \textit{Finite Markov Chains}, Springer-Verlag, New York (1976).
%
\bibitem{KirklandZeng} S. Kirkland, Z. Zeng, Kemeny's Constant And An Analogue Of Braess' Paradox For Trees, Electronic Journal of Linear Algebra 31(1) (2016): 444--464.
%
\bibitem{Lovasz} L. Lov\'asz, Random walks on graphs: A survey, Combinatorics, Paul Erd\H{o}s Is Eighty 2(1), edited by D. Mikl\'os, D. S\'os, and T. Sz\"oni, Budapest, Hungary, U\'anos Bolyai Mathematical Society (1993): 1--46.

\bibitem{Meyn} S.P. Meyn, \textit{Control Techniques for Complex Networks}, Cambridge University Press (2008).
%
\bibitem{Seneta} E. Seneta, \textit{Non-negative matrices and Markov chains}, 2nd rev. ed., Springer, New York (1981).
%
%
\end{thebibliography}
\end{document}